\documentclass[preprint,12pt]{elsarticle}
\usepackage[T1]{fontenc}
\usepackage[utf8]{inputenc}
\usepackage{amsmath,amssymb,amsthm,mathtools}
\usepackage{booktabs}
\usepackage{array}
\usepackage{tikz}
\usepackage{enumitem}
\usepackage{microtype}
\usepackage[hidelinks]{hyperref}
\usepackage{newunicodechar}
\newunicodechar{，}{,}

\biboptions{sort&compress}
\newtheorem{theorem}{Theorem}[section]
\newtheorem{lemma}[theorem]{Lemma}
\newtheorem{proposition}[theorem]{Proposition}
\newtheorem{corollary}[theorem]{Corollary}
\newtheorem{definition}[theorem]{Definition}
\newtheorem{example}[theorem]{Example}
\newtheorem{remark}[theorem]{Remark}

\begin{document}

\begin{frontmatter}

\title{\textbf{Quasi-overlap and quasi-grouping functions on bounded pseudo-ordered sets}
\tnoteref{funding}}

\tnotetext[funding]{Supported by the National Natural Science Foundation of China (No. 12471440).}

\author{Xue-ping Wang\corref{cor1}}
\ead{xpwang1@hotmail.com}

\author{Hong-fei Liu}
\ead{liuhfmath@163.com}

\cortext[cor1]{Corresponding author.}
\address{School of Mathematical Sciences, Sichuan Normal University, Chengdu 610066, Sichuan, People's Republic of China}

\begin{abstract}
In this article, we first introduce quasi-overlap and quasi-grouping functions on bounded pseudo-ordered sets, respectively, and give their respective preliminary constructions and the induced quasi-overlap functions on quotient posets.. We then explore the invariant properties under weighting and automorphisms. We also introduce the migrativity, the homogeneity, the idempotency, the cancellation law, the Archimedean and limiting properties of quasi-overlap functions and discuss their relations. We finally present two constructing methods of quasi-overlap functions, which are based on boundary-faithful mappings and adjunctions on bounded psosets and the transitive-endpoints of a bounded trellis, respectively. The corresponding statements are valid for quasi-grouping functions as well.
\end{abstract}


\end{frontmatter}

\section{Introduction}

Aggregation functions play an important role in fuzzy set theory, approximate reasoning, information fusion, and decision making \cite{Beliakov2007,Grabisch2009}. Among them, triangular norms (t-norms for short) provide the standard associative model of fuzzy conjunction \cite{Klement2000,Menger1942,Schweizer1983}. Overlap functions are introduced to model the degree of overlap between fuzzy classes, with image processing as a motivating application \cite{Bustince2010}. Unlike t-norms, overlap functions are not required to be associative. Their disjunctive counterparts, grouping functions, are subsequently studied in fuzzy modelling and pairwise comparison \cite{Bustince2012,Bedregal2013}. These notions are originally defined on $[0,1]$ and are later extended to more general ordered structures. In particular, Paiva et al.\ \cite{Paiva2021} introduced overlap and quasi-overlap functions on bounded lattices, while Qiao studied overlap and grouping functions on complete lattices \cite{Qiao2021} and also investigated the constructions of quasi-overlap functions on bounded partially ordered sets \cite{Qiao2022}. Related lattice-valued developments can be found in \cite{Ertugrul2015, QiaoZhao2022,Ouyang2021,SunPangZhang2022,Monteiro2024}.

A different situation arises when the transitivity of the underlying relation is dropped. Trellises, introduced by Skala \cite{Skala1971,Skala1972}, are lattice-like structures whose underlying relation is reflexive and antisymmetric but need not be transitive, called a pseudo-ordered set. This setting is relevant, for example, to cyclic
preferences and cyclic competition \cite{Fishburn1970,Kerr2002}. Aggregation operators on bounded trellises have recently been studied for t-norms \cite{Zedam2023,Geng2026}, uninorms and nullnorms \cite{Kong2024,Xiu2025,JiangWangLiu2025,KongLiu2026}.

Observe that the loss of transitivity, however, prevents a direct transfer of many lattice arguments: meet and join need not be increasing or associative, and separate monotonicity need not imply increasingness with respect to the product pseudo-order. Therefore, one of the most important tasks when we investigate aggregation functions on bounded pseudo-orders is to understand how dropping the transitivity property affects such operators. On the other hand, quasi-overlap functions along with their generalized forms will not only help us obtain new Choquet integral classes but also possibly help us to establish new logical systems on a more generally ordered sets rather than bounded lattices \cite{Qiao2022}. Motivated by the above two aspects, we consider the compatibility of overlap and grouping typical functions and the partially ordered sets without the transitivity. More specifically, in this article, we introduce quasi-overlap and quasi-grouping functions on bounded pseudo-ordered sets and develop their theory.

The article is organized as follows. In Section 2, we recall some necessary notions and results on pseudo-ordered sets, trellises and aggregation operators. In Section 3 we introduce quasi-overlap and quasi-grouping functions on bounded pseudo-ordered sets and show their preliminary constructions. In Section 4, we study the main algebraic properties of quasi-overlap and quasi-grouping functions and the relations among them. In Section 5, we develop two construction methods. A concluding remark is drawn in Section 6.

\section{Preliminaries}\label{sec:prelim}
This section gives some known concepts and results that will be used in the sequel.
\begin{definition}[\cite{Birkhoff73,Skala1971}]
\emph{A binary relation $\trianglelefteq$ on a nonempty set $X$ is called a \emph{pseudo-order} on $X$ if, for all $x,y\in X$,
\begin{enumerate}[label=(\roman*)]
\item $x\trianglelefteq x$ (reflexive);
\item $x\trianglelefteq y$ and $y\trianglelefteq x$ imply $x=y$ (antisymmetric).
\end{enumerate}
The pair $(X,\trianglelefteq)$ is called a \emph{pseudo-ordered set} (\emph{psoset} for short). The psoset $(X,\trianglelefteq)$ is called a \emph{partially ordered set} (\emph{poset} for short) if  for any $x, y,z\in X$, $x\trianglelefteq y$ and $y\trianglelefteq z$ imply $x\trianglelefteq z$ (transitive). A poset $(X,\trianglelefteq)$ is called a \emph{totally ordered set} if for any $x, y\in X$, $x\trianglelefteq y$ or $y\trianglelefteq x$. A psoset that is not a poset is called a \emph{proper psoset}}.
\end{definition}

Let $(X,\trianglelefteq)$ be a psoset, and $x,y\in X$. If $x\trianglelefteq y$ and $x\neq y$, we write $x\triangleleft y$. $x$ and $y$ are called \emph{incomparable} if neither $x\trianglelefteq y$ nor $y\trianglelefteq x$ hold, and in this case we use the symbol $x\parallel y$. For $x,y\in X$, we write $x\precsim y$ if there is a finite sequence $(x_0, x_1,\cdots, x_n)$ of elements from $X$ such that
$$
 x=x_0\trianglelefteq x_1\trianglelefteq\cdots\trianglelefteq x_n=y.
$$
Clearly, $\precsim$ is reflexive and transitive. A nonempty subset $C\subseteq X$ is a \emph{cycle} if, for every $x,y\in C$, there are finite sequences $(x_0, x_1,\cdots, x_m)$ and $(y_0, y_1,\cdots, y_n)$ whose terms all belong to $C$ such that
$$
 x=x_0\trianglelefteq x_1\trianglelefteq\cdots\trianglelefteq x_m=y
 \mbox{ and }
 y=y_0\trianglelefteq y_1\trianglelefteq\cdots\trianglelefteq y_n=x.
$$
By the antisymmetry, every non-trivial cycle contains at least three elements.

\begin{definition}[\cite{Skala1971}]
\emph{A psoset $(X,\trianglelefteq)$ is a $\wedge$-\emph{semi-trellis} if every pair $x,y\in X$ has a greatest lower bound, denoted $x\wedge y$. Dually, it is a $\vee$-\emph{semi-trellis} if every pair $x,y\in X$ has a least upper bound, denoted $x\vee y$. A \emph{trellis} is both a $\wedge$-semi-trellis and a $\vee$-semi-trellis, in symbols $\mathbb{T}=(X,\trianglelefteq,\wedge,\vee)$. It is \emph{bounded} if it has a smallest element $0$ and a greatest element $1$, i.e., $0\trianglelefteq x\trianglelefteq1$ for any $x\in X$ and is denoted by $\mathbb{T}=(X,\trianglelefteq,\wedge,\vee,0,1).$
A trellis that is not a lattice is called a \emph{proper trellis}. A nonempty subset $R$ of $X$ is called a \emph{subtrellis} of $\mathbb{T}$ if $(R,\trianglelefteq,\wedge,\vee)$ is also a trellis.}
\end{definition}

\begin{example}\label{ex:T5}
\emph{Let $X=\{0,a,b,c,1\}$. Besides reflexivity, define
$$
 0\trianglelefteq x\trianglelefteq 1 \mbox{ for any }x\in X, a\trianglelefteq b, b\trianglelefteq c, c\trianglelefteq a.
$$
Then $(X,\trianglelefteq)$ is a psoset, and $a,b,c$ form a cycle. Every pair of $X$ has a meet and a join, for example,
$$
 a\wedge b=a, a\vee b=b,
 b\wedge c=b, b\vee c=c,
 c\wedge a=c, c\vee a=a.
$$
The resulting structure $\mathbb{T}_5=(X,\trianglelefteq,\wedge,\vee,0,1)$ is a bounded proper trellis (see Figure 1).}

\begin{figure}[ht]
\centering
\begin{tikzpicture}[scale=1.0]
\node[circle,draw,inner sep=2pt] (0) at (0,0) {$0$};
\node[circle,draw,inner sep=2pt] (a) at (-1,1.5) {$a$};
\node[circle,draw,inner sep=2pt] (b) at (0,1.9) {$b$};
\node[circle,draw,inner sep=2pt] (c) at (1,1.5) {$c$};
\node[circle,draw,inner sep=2pt] (1) at (0,3.3) {$1$};
\draw (0)--(a); \draw (0)--(b); \draw (0)--(c);
\draw (a)--(1); \draw (b)--(1); \draw (c)--(1);
\draw[->,bend left=20] (a) to (b);
\draw[->,bend left=20] (b) to (c);
\draw[->,bend left=20] (c) to (a);
\end{tikzpicture}
\caption{The bounded proper trellis $\mathbb{T}_5$.}
\label{fig:T5}
\end{figure}
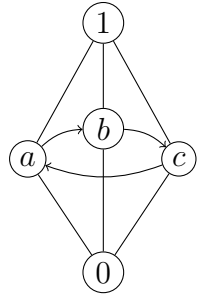
\end{example}
\begin{theorem}[\cite{Skala1971}]\label{thm:skala-equivalence}
For a trellis $\mathbb{T}=(X,\trianglelefteq,\wedge,\vee)$, the following are equivalent:
\begin{enumerate}[label=(\roman*)]
\item $\trianglelefteq$ is transitive;
\item $\wedge$ is associative;
\item $\vee$ is associative;
\item $\mathbb{T}$ is a lattice.
\end{enumerate}
\end{theorem}

Let $(X,\trianglelefteq)$ be a psoset. Define a product pseudo-order $\trianglelefteq_2$ on $X^2$ by
$$(x,y)\trianglelefteq_2(x',y')\Longleftrightarrow x\trianglelefteq x'\mbox{ and }y\trianglelefteq y'.
$$

\begin{definition}[\cite{Zedam2023}]
A binary operation $F:X^2\to X$ on a psoset $(X,\trianglelefteq)$ is \emph{increasing} if, for all $x,y,x',y'\in X$,
$$
 (x,y)\trianglelefteq_2(x',y')\Longrightarrow F(x,y)\trianglelefteq F(x',y').
$$
\end{definition}

Note that an increasing operation is increasing in each coordinate, and the converse is valid on posets because two one-coordinate inequalities can be concatenated by transitivity. However, it may fail on a proper psoset as shown by the following example.

\begin{example}
\emph{Let $c_0=a,c_1=b,c_2=c$, with indices read modulo $3$, and define $F:\mathbb{T}_5^2\to\mathbb{T}_5$ by
$$
F(x,y)=
\begin{cases}
0, & x=0\mbox{ or }y=0,\\
1, & x=1\mbox{ or }y=1,\mbox{ and }x,y\neq0,\\
c_{i+j}, & x=c_i,\ y=c_j.
\end{cases}
$$
For a fixed $c_j$, the relation $c_i\trianglelefteq c_{i+1}$ is sent to
$c_{i+j}\trianglelefteq c_{i+j+1}$, and the boundary cases are immediate. Hence $F$ is increasing in each coordinate. Nevertheless,
$$(a,a)\trianglelefteq_2(b,b), F(a,a)=a, F(b,b)=c, \mbox{and }a\ntrianglelefteq c.
$$
Thus $F$ is not increasing with respect to $\trianglelefteq_2$.}
\end{example}

The following example illustrates that both meet and join operations of a trellis need not be increasing.
\begin{example}\label{exam2.7}
\emph{On $\mathbb{T}_5$, it is easy to see that
$$
 (a,c)\trianglelefteq_2(b,c),
$$
but
$$
 a\wedge c=c\ntrianglelefteq b=b\wedge c, a\vee c=a\ntrianglelefteq c=b\vee c.
$$
Thus, neither $\wedge$ nor $\vee$ is increasing.}
\end{example}

\begin{definition}[\cite{Skala1972}]
\emph{Let $\mathbb{T}=(X,\trianglelefteq,\wedge,\vee)$ be a trellis. An element $a\in X$ is called
\begin{enumerate}[label=(\roman*)]
\item \emph{right-transitive} if, for all $x,y \in X$, $a\trianglelefteq x\trianglelefteq y$ implies $a\trianglelefteq y$;
\item \emph{left-transitive} if, for all $x,y \in X$, $x\trianglelefteq y\trianglelefteq a$ implies $x\trianglelefteq a$;
\item \emph{middle-transitive} if, for all $x,y \in X$, $x\trianglelefteq a\trianglelefteq y$ implies $x\trianglelefteq y$;
\item \emph{transitive} if it is right-, left- and middle-transitive.
\end{enumerate}
The corresponding sets of such elements are denoted by $X^{\rm rtr}$, $X^{\rm ltr}$, $X^{\rm mtr}$, and $X^{\rm tr}$, respectively.}
\end{definition}

In Example~\ref{ex:T5}, $X^{\rm rtr}$=$X^{\rm ltr}$=$X^{\rm mtr}$=\{0,1\}, and $X^{\rm tr}=\{0,1\}$.
\begin{definition}[\cite{Zedam2023}]
\emph{Let $(X,\trianglelefteq,0,1)$ be a bounded psoset. A binary operation $T:X^2\to X$ is a \emph{t-norm} if it is increasing, commutative, associative, and has the neutral element $1$. It is \emph{positive} if
$$
 T(x,y)=0\Longleftrightarrow x=0\mbox{ or }y=0.
$$
Dually, a \emph{t-conorm} $S:X^2\to X$ is increasing, commutative, associative, and has the neutral element $0$. It is positive if
$$
 S(x,y)=1\Longleftrightarrow x=1\mbox{ or }y=1.
$$}
\end{definition}

\section{Quasi-overlap and quasi-grouping functions}\label{sec:definitions}
In this section, we introduce the concepts of quasi-overlap and quasi-grouping functions on a bounded psoset and give their respective preliminary constructions on a bounded trellis. We also discuss their relation as well as the induced quasi-overlap functions on quotient posets.
\begin{definition}\label{def:qo}
\emph{Let $\mathbb{P}=(X,\trianglelefteq,0,1)$ be a bounded psoset. A binary operation $O:X^2\to X$ is a \emph{quasi-overlap function} on $\mathbb{P}$ if, for all $x,y, x',y'\in X$,
\begin{enumerate}[label=(QO\arabic*)]
\item $O(x,y)=O(y,x)$;
\item $O(x,y)=0$ if and only if $x=0$ or $y=0$;
\item $O(x,y)=1$ if and only if $x=y=1$;
\item $(x,y)\trianglelefteq_2(x',y')$ implies $O(x,y)\trianglelefteq O(x',y')$.
\end{enumerate}
The set of all quasi-overlap functions on $\mathbb{P}$ is denoted by $\mathbb{QO}(\mathbb{P})$. When the underlying bounded psoset $\mathbb{P}$ is a bounded trellis $\mathbb{T}$, we also write $\mathbb{QO}(\mathbb{T})$.}
\end{definition}

The definition of a quasi-grouping function $G:X^2\to X$ on a bounded psoset $\mathbb{P}=(X,\trianglelefteq,0,1)$ can dually be obtained from that of a quasi-overlap function, exactly replacing (QO2) and (QO3) by

(QG2) $G(x,y)=0$ if and only if $x=y=0$, \mbox{and}

(QG3) $G(x,y)=1$ if and only if $x=1$ or $y=1$,\\
 respectively.
The set of all quasi-grouping functions on $\mathbb{P}$ is denoted by $\mathbb{QG}(\mathbb{P})$. When the underlying bounded psoset $\mathbb{P}$ is a bounded trellis $\mathbb{T}$, we also write $\mathbb{QG}(\mathbb{T})$.

The following example gives all quasi-overlap and quasi-grouping functions on the bounded proper trellis $\mathbb{T}_5$, respectively.
\begin{example}\label{examp:T5-classification}
\emph{Let $\mathbb{T}_5$ be the trellis in Example~\ref{ex:T5}, $C=\{a,b,c\}$ and let $\sigma:C\to C$ be the cyclic permutation
$$
 \sigma(a)=b,\sigma(b)=c,\sigma(c)=a.
$$
For $\eta\in C$ and $\theta\in\{\eta,\sigma(\eta)\}$, define
\begin{enumerate}[label=(\roman*)]
\item $$
O_{\eta,\theta}(x,y)=
\begin{cases}
0,&x=0\mbox{ or }y=0,\\
1,&x=y=1,\\
\theta,&\{x,y\}\cap\{1\}\neq\emptyset\mbox{ and }\{x,y\}\cap C\neq\emptyset,\\
\eta,&x,y\in C.
\end{cases}
$$
\item $$
G_{\eta,\theta}(x,y)=
\begin{cases}
0,&x=y=0,\\
1,&x=1\mbox{ or }y=1,\\
\eta,&\{x,y\}\cap\{0\}\neq\emptyset\mbox{ and }\{x,y\}\cap C\neq\emptyset,\\
\theta,&x,y\in C.
\end{cases}
$$
\end{enumerate}
Then
$$
 \mathbb{QO}(\mathbb{T}_5)=\{O_{\eta,\theta}:\eta\in C,\ \theta\in\{\eta,\sigma(\eta)\}\}
$$
and
$$
 \mathbb{QG}(\mathbb{T}_5)=\{G_{\eta,\theta}:\eta\in C,\ \theta\in\{\eta,\sigma(\eta)\}\}.
$$
In particular, $|\mathbb{QO}(\mathbb{T}_5)|=|\mathbb{QG}(\mathbb{T}_5)|=6$ (see Tables \ref{tab:Oeta} and \ref{tab:Geta}).
\begin{table}[ht]
\centering
\caption{The operation $O_{\eta,\theta}$ on $\mathbb{T}_5$.}
\label{tab:Oeta}
\begin{tabular}{c|ccccc}
$O_{\eta,\theta}$&$0$&$a$&$b$&$c$&$1$\\\hline
$0$&$0$&$0$&$0$&$0$&$0$\\
$a$&$0$&$\eta$&$\eta$&$\eta$&$\theta$\\
$b$&$0$&$\eta$&$\eta$&$\eta$&$\theta$\\
$c$&$0$&$\eta$&$\eta$&$\eta$&$\theta$\\
$1$&$0$&$\theta$&$\theta$&$\theta$&$1$\\
\end{tabular}
\end{table}}

\begin{table}[ht]
\centering
\caption{The operation $G_{\eta,\theta}$ on $\mathbb{T}_5$.}
\label{tab:Geta}
\begin{tabular}{c|ccccc}
$G_{\eta,\theta}$&$0$&$a$&$b$&$c$&$1$\\\hline
$0$&$0$&$\eta$&$\eta$&$\eta$&$1$\\
$a$&$\eta$&$\theta$&$\theta$&$\theta$&$1$\\
$b$&$\eta$&$\theta$&$\theta$&$\theta$&$1$\\
$c$&$\eta$&$\theta$&$\theta$&$\theta$&$1$\\
$1$&$1$&$1$&$1$&$1$&$1$\\
\end{tabular}
\end{table}
\end{example}
\begin{proof}
We only prove (i), (ii) being dual. Let $O\in\mathbb{QO}(\mathbb{T}_5)$. Conditions (QO2) and (QO3) in Definition \ref{def:qo} determine the first row and column and give $O(1,1)=1$. Every other value of $O$ belongs to $C$. We determine the remaining values of $O$ in two steps.

Step 1. For $x,y\in C$, put $u_{xy}=O(x,y)$. Clearly, $u_{xy}\in C$. Since
$$
 (a,b)\trianglelefteq_2(a,c)
 \mbox{ and }
 (c,a)\trianglelefteq_2(a,b),
$$
the commutativity and increasingness give
$u_{ab}\trianglelefteq u_{ac}=u_{ca}\trianglelefteq u_{ab}$. By the antisymmetry, $u_{ab}=u_{ac}=u_{ca}$. Moreover,
$$
 (a,b)\trianglelefteq_2(b,c)\trianglelefteq_2(c,a),
$$
so $u_{ab}=u_{bc}=u_{ca}$. Thus by the commutativity of $O$, we have $u_{ab}=u_{ba}=u_{bc}=u_{cb}=u_{ca}=u_{ac}$. Denote $u_{ac}=\eta$. Since
$(c,a)\trianglelefteq_2(a,a)\trianglelefteq_2(a,b)$, we obtain $u_{aa}=\eta$. Similarly, $u_{bb}=u_{cc}=\eta$. Therefore,
$O(x,y)=\eta\mbox{ for all }x,y\in C$.

Step 2.  For $x\in C$, set $v_x=O(x,1)$. Since $(x,x)\trianglelefteq_2(x,1)$, the increasingness gives $\eta\trianglelefteq v_x$. As $v_x\in C$, it follows that
$v_x\in\{\eta,\sigma(\eta)\}$. Furthermore, $x\trianglelefteq\sigma(x)$ implies $v_x\trianglelefteq v_{\sigma(x)}$. If $v_x=\sigma(\eta)$, then $v_{\sigma(x)}=\sigma(\eta)$ since $\sigma(\eta)\ntrianglelefteq\eta$, and the arbitrariness of $x$ gives $v_a=v_b=v_c=\sigma(\eta)$. Otherwise, $v_a=v_b=v_c=\eta$.

Therefore, $O\in \{O_{\eta,\theta}:\eta\in C,\ \theta\in\{\eta,\sigma(\eta)\}\}$.

Conversely, we prove that every $O_{\eta,\theta}$ is a quasi-overlap function. The conditions (QO1)--(QO3) follow directly from the definition of $O_{\eta,\theta}$. It remains to verify the increasingness of $O_{\eta,\theta}$. Suppose that $(x,y)\trianglelefteq_2(x',y')$. We prove $O_{\eta,\theta}(x,y)\trianglelefteq O_{\eta,\theta}(x',y')$ by four cases as below.

(i) If $x=0$ or $y=0$, then by the definition of $O_{\eta,\theta}(x,y)$, $O_{\eta,\theta}(x,y)=0$, which implies $O_{\eta,\theta}(x,y)\trianglelefteq O_{\eta,\theta}(x',y')$.

(ii) If $(x,y)=(1,1)$, then clearly $x'=y'=1$, so $O_{\eta,\theta}(x,y)=1= O_{\eta,\theta}(x',y')$ by the definition of $O_{\eta,\theta}(x,y)$.

(iii) If $x,y\in C$, then by the definition of $O_{\eta,\theta}(x,y)$, $O_{\eta,\theta}(x,y)=\eta$. The pair $(x',y')$ must lie in $C^2\cup(C\times\{1\})\cup(\{1\}\times C)\cup\{(1,1)\}$. This follows that $$O_{\eta,\theta}(x',y')\in\{\eta,\theta,1\}\mbox{ with }\eta\trianglelefteq\theta\trianglelefteq1.$$
Thus $O_{\eta,\theta}(x,y)\trianglelefteq O_{\eta,\theta}(x',y')$.

(iv) If one coordinate of $(x,y)$ is $1$ and the other belongs to $C$, say $x=1$ and $y\in C$, then $x'=1$. Hence $O_{\eta,\theta}(x',y')\in\{\theta,1\}$, and
$$
 O_{\eta,\theta}(x,y)=\theta\trianglelefteq O_{\eta,\theta}(x',y').
$$

Therefore, $O_{\eta,\theta}\in\mathbb{QO}(\mathbb{T}_5)$ by Definition \ref{def:qo}.

It is easy to see $|\mathbb{QO}(\mathbb{T}_5)|=6$.
\end{proof}
\begin{remark}\label{remak:qotnorm}
\emph{Let $\mathbb{P}=(X,\trianglelefteq,0,1)$ be a bounded psoset.
\begin{enumerate}[label=(\roman*)]
\item Every positive t-norm on $\mathbb{P}$ is a quasi-overlap function.
\item A quasi-overlap function on $\mathbb{P}$ is a positive t-norm if and only if it is associative and has the neutral element $1$.
\item Every positive t-conorm on $\mathbb{P}$ is a quasi-grouping function.
\item A quasi-grouping function on $\mathbb{P}$ is a positive t-conorm if and only if it is associative and has the neutral element $0$.
\end{enumerate}}
\end{remark}

Note that in general, the associativity of Remark~\ref{remak:qotnorm} (ii) (resp. Remark~\ref{remak:qotnorm} (iv)) is not enough. For example, the quasi-overlap function $O_{a,a}$ on $\mathbb{T}_5$ is associative, but
$$
 O_{a,a}(b,1)=a\neq b,
$$
so $1$ is not a neutral element of $O_{a,a}$.

\begin{proposition}\label{prop:constant-qo}
Let $\mathbb{P}=(X,\trianglelefteq,0,1)$ be a bounded psoset, $\eta\in X\setminus\{0,1\}$. Define $O_\eta$ and $G_\eta:X^2\to X$ by
$$
O_\eta(x,y)=
\begin{cases}
0,&x=0\mbox{ or }y=0,\\
1,&x=y=1,\\
\eta,&\mbox{otherwise}
\end{cases}\mbox{ and }
G_\eta(x,y)=
\begin{cases}
0,&x=y=0,\\
1,&x=1\mbox{ or }y=1,\\
\eta,&\mbox{otherwise},
\end{cases}
$$respectively.
Then $O_\eta\in\mathbb{QO}(\mathbb{P})$ and $G_\eta\in\mathbb{QG}(\mathbb{P})$.
\end{proposition}

\begin{proof}
We only prove the assertion for $O_\eta$ since the proof for $G_\eta$ is dual. Conditions (QO1)--(QO3) follow directly from Definition \ref{def:qo}. It remains to verify (QO4). For all $x,y, x',y'\in X$, let $(x,y)\trianglelefteq_2(x',y')$. We prove $O_\eta(x,y)\trianglelefteq O_\eta(x',y')$ by distinguishing three cases as follows.

If $O_\eta(x,y)=0$, then clearly, $O_\eta(x,y)\trianglelefteq O_\eta(x',y')$.

If $O_\eta(x,y)=1$, then by (QO3), $x=y=1$. Hence $1\trianglelefteq x'$ and $1\trianglelefteq y'$, which together with the antisymmetry yield $x'=y'=1$. Thus $O_\eta(x,y)\trianglelefteq O_\eta(x',y')$.

Now, suppose that $O_\eta(x,y)=\eta$. Then neither $x'=0$ nor $y'=0$. Indeed, $x'=0$ would imply $x\trianglelefteq0$, and then $x=0$. Thus by (QO2), $O_\eta(x,y)=0$, a contradiction. Similarly, $y'=0$ would imply a contradiction. Hence $O_\eta(x',y')\in\{\eta,1\}$, which leads to $O_\eta(x,y)\trianglelefteq O_\eta(x',y')$.
\end{proof}

 Note that one can easily check that in Example \ref{examp:T5-classification}, taking $\eta=a$ and $\theta=b$ gives a quasi-overlap $O_{a,b}$, which cannot be obtained by Proposition \ref{prop:constant-qo}.
\begin{definition}
\emph{Let $\mathbb{P}=(X,\trianglelefteq,0,1)$ be a bounded psoset. A unary operation $N:X\to X$ is an \emph{involutive negation} on $\mathbb{P}$ if,  for all $x,y\in X$,
\begin{enumerate}[label=(\roman*)]
\item $x\trianglelefteq y$ implies $N(y)\trianglelefteq N(x)$;
\item $N(0)=1$ and $N(1)=0$;
\item $N(N(x))=x$.
\end{enumerate}}
\end{definition}

\begin{example}\label{ex:negations}
\emph{All the involutive negations $N_a$, $N_b$ and $N_c$ on $\mathbb{T}_5$ are shown by Table \ref{tab:Oeta001}.}
\end{example}
\begin{table}[ht]
\centering
\caption{All the involutive negations on $\mathbb{T}_5$.}
\label{tab:Oeta001}
\begin{tabular}{c|ccccc}
&$0$&$a$&$b$&$c$&$1$\\\hline
$N_a$&$1$&$a$&$c$&$b$&$0$\\
$N_b$&$1$&$c$&$b$&$a$&$0$\\
$N_c$&$1$&$b$&$a$&$c$&$0$
\end{tabular}
\end{table}
\begin{theorem}[De Morgan duality]\label{thm:duality}
Let $N$ be an involutive negation on a bounded psoset $\mathbb{P}=(X,\trianglelefteq,0,1)$.
\begin{enumerate}[label=(\roman*)]
\item If $O\in\mathbb{QO}(\mathbb{P})$, then the binary operation $G_N:X^2\to X$ defined by
$$
 G_N(x,y)=N\bigl(O(N(x),N(y))\bigr)
$$
belongs to $\mathbb{QG}(\mathbb{P})$.
\item If $G\in\mathbb{QG}(\mathbb{P})$, then the binary operation $O_N:X^2\to X$ defined by
$$
 O_N(x,y)=N\bigl(G(N(x),N(y))\bigr)
$$
belongs to $\mathbb{QO}(\mathbb{P})$.
\end{enumerate}
\end{theorem}

\begin{proof}
We only prove (i), the proof of (ii) being dual. The commutativity of $G_N$ follows from that of $O$. Moreover,
$$
\begin{aligned}
G_N(x,y)=0
&\Longleftrightarrow O(N(x),N(y))=1\\
&\Longleftrightarrow N(x)=N(y)=1 \mbox{ by (QO3)}\\
&\Longleftrightarrow x=y=0,
\end{aligned}
$$
and
$$
\begin{aligned}
G_N(x,y)=1
&\Longleftrightarrow O(N(x),N(y))=0\\
&\Longleftrightarrow N(x)=0\mbox{ or }N(y)=0 \mbox{ by (QO2)}\\
&\Longleftrightarrow x=1\mbox{ or }y=1.
\end{aligned}
$$
It remains to prove the increasingness. Let $x\trianglelefteq x'$ and $y\trianglelefteq y'$. Since $N$ is decreasing,
$$
 N(x')\trianglelefteq N(x), N(y')\trianglelefteq N(y).
$$
By the increasingness of $O$,
$$
 O(N(x'),N(y'))\trianglelefteq O(N(x),N(y)).
$$
Applying $N$ to the last inequality gives
$$
 G_N(x,y)\trianglelefteq G_N(x',y').
$$
Hence $G_N\in\mathbb{QG}(\mathbb{P})$.
\end{proof}

Notice that from Theorem~\ref{thm:duality}, each of negations a bounded psoset transforms quasi-overlap functions into quasi-grouping functions, and vice versa. Therefore, in what follows, all discussions on quasi-overlap functions are true for quasi-grouping functions dually.

In the rest of this section, we investigate induced quasi-overlap functions on quotient posets. We first have the following two lemmas.
\begin{lemma}\label{lem:path-preservation}
Let $(X,\trianglelefteq)$ be a psoset and $F:X^2\to X$ be increasing. If $x\precsim x'$ and $y\precsim y'$, then
$$
 F(x,y)\precsim F(x',y').
$$
\end{lemma}

\begin{proof}
Because $x\precsim x'$ and $y\precsim y'$, let
$$
 x=x_0\trianglelefteq x_1\trianglelefteq\cdots\trianglelefteq x_m=x'
$$
and
$$
 y=y_0\trianglelefteq y_1\trianglelefteq\cdots\trianglelefteq y_n=y'.
$$
We then have
$$
 F(x_i,y_0)\trianglelefteq F(x_{i+1},y_0)
$$ for every $i\in\{0,\ldots,m-1\}$ since $F$ is increasing and $y_0\trianglelefteq y_0$.
Analogously,
$$
 F(x_m,y_j)\trianglelefteq F(x_m,y_{j+1})
$$ for every $j\in\{0,\ldots,n-1\}$ since $x_m\trianglelefteq x_m$. Therefore,
$$
 F(x,y)=F(x_0,y_0)\precsim F(x_m,y_0)\precsim F(x_m,y_n)=F(x',y'),
$$
which proves the assertion.
\end{proof}

For $x,y\in X$, define $x\equiv y$ if $x\precsim y$ and $y\precsim x$. It is easy to see that $\equiv$ is an equivalence relation on $X$.

\begin{lemma}\label{lemm:scc-operation}
Let $(X,\trianglelefteq)$ be a psoset and let $F:X^2\to X$ be increasing. If $x\equiv x'$ and $y\equiv y'$ with $x, y, x', y'\in X$, then
$$
 F(x,y)\equiv F(x',y').
$$
Consequently, $F$ induces a well-defined increasing operation $\overline F$ on $X/{\equiv}$ by
$$
 \overline F([x],[y])=[F(x,y)].
$$
Moreover, the relation $\leq$ defined by
$$
 [x]\leq[y]\Longleftrightarrow x\precsim y
$$
is a partial order on $X/{\equiv}$.
\end{lemma}

\begin{proof}
Suppose that $x\equiv x'$ and $y\equiv y'$. Then $x\precsim x'$, $x'\precsim x$, $y\precsim y'$ and $y'\precsim y$. By Lemma~\ref{lem:path-preservation}, we have
$$
 F(x,y)\precsim F(x',y')\mbox{ and } F(x',y')\precsim F(x,y).
$$
Hence $F(x,y)\equiv F(x',y')$, and therefore $\overline F$ is well defined.

We next verify that $\leq$ is well defined. Suppose that $[x]=[\widetilde x]$, $[y]=[\widetilde y]$, and $[x]\leq[y]$. Then $x\precsim y$ and
$$
 \widetilde x\precsim x\precsim y\precsim\widetilde y,
$$
so $\widetilde x\precsim\widetilde y$, i.e., $[\widetilde x]\leq[\widetilde y]$.

The reflexivity and transitivity of $\leq$ follow from the corresponding properties of $\precsim$. If $[x]\leq[y]$ and $[y]\leq[x]$, then $x\equiv y$, and hence $[x]=[y]$, i.e., $\leq$ is antisymmetric. Therefore, $(X/{\equiv},\leq)$ is a poset.

Finally, let $[x]\leq[x']$ and $[y]\leq[y']$. Then $x\precsim x'$ and $y\precsim y'$, so Lemma~\ref{lem:path-preservation} yields
$$
 F(x,y)\precsim F(x',y'),
$$
then
$$
[F(x,y)]\leq [F(x',y')].
$$
Therefore, $\overline F([x],[y])\leq \overline F([x'],[y'])$, i.e., $\overline F$ is increasing.
\end{proof}

The next result shows that a quasi-overlap function induces a well-defined operation on a quotient poset in a natural way, which is also a quasi-overlap function on the quotient poset.
\begin{theorem}\label{thm:quotient-qo}

Let $\mathbb{P}=(X,\trianglelefteq,0,1)$ be a bounded psoset and $O\in\mathbb{QO}(\mathbb{P})$. Then the induced operation
$$
 \overline O([x],[y])=[O(x,y)]
$$
is a quasi-overlap function on the bounded poset $(X/{\equiv},\leq,[0],[1])$.
\end{theorem}

\begin{proof}
By Lemma~\ref{lemm:scc-operation}, $\overline O$ is well defined, commutative, and increasing. We first show that
$$
 [0]=\{0\}\mbox{ and }[1]=\{1\}.
$$
Indeed, suppose that $z\precsim0$ with $z\in X$. Then there is a finite sequence $(z_0, z_1,\cdots, z_m)$ of elements from $X$ such that
$$
 z=z_0\trianglelefteq z_1\trianglelefteq\cdots\trianglelefteq z_m=0.
$$
Since $z_{m-1}\trianglelefteq0$ and $0$ is the smallest element, the antisymmetry gives $z_{m-1}=0$. Repeating the above argument yields $z=0$. Hence $[0]=\{0\}$. Dually, if $1\precsim z$, then $[1]=\{1\}$.
Thus $$
\begin{aligned}
\overline O([x],[y])=[0]
&\Longleftrightarrow [O(x,y)]=[0]\\
&\Longleftrightarrow O(x,y)=0\\
&\Longleftrightarrow x=0\mbox{ or }y=0 \mbox{ by (QO2)}\\
&\Longleftrightarrow [x]=[0]\mbox{ or }[y]=[0]
\end{aligned}
$$
and
$$
\begin{aligned}
\overline O([x],[y])=[1]
&\Longleftrightarrow [O(x,y)]=[1]\\
&\Longleftrightarrow O(x,y)=1\\
&\Longleftrightarrow x=y=1 \mbox{ by (QO3)}\\
&\Longleftrightarrow [x]=[y]=[1].
\end{aligned}
$$
Therefore, $\overline O$ is a quasi-overlap function on $X/{\equiv}$ by Definition \ref{def:qo}.
\end{proof}

\begin{example}\label{ex:T5-quotient}
\upshape
One can check that the quotient of $\mathbb{T}_5$ under $\equiv$ has three elements
$$
 \mathbf 0=[0], \mathbf c=[a]=[b]=[c], \mathbf 1=[1],
$$
which form the chain $\mathbf0<\mathbf c<\mathbf1$. From Example~\ref{examp:T5-classification} every quasi-overlap function induces the same quotient operation $\overline O$ (see Table 4).
\begin{table}[ht]
\centering
\caption{The quasi-overlap function on $X/{\equiv}$}
\label{tab:Oeta0}
\begin{tabular}{c|ccccc}
$\overline O$ & $\mathbf 0$ & $\mathbf c$ & $\mathbf 1$\\
\hline
$\mathbf 0$ & $\mathbf 0$ & $\mathbf 0$ & $\mathbf 0$\\
$\mathbf c$ & $\mathbf 0$ & $\mathbf c$ & $\mathbf c$\\
$\mathbf 1$ & $\mathbf 0$ & $\mathbf c$ & $\mathbf 1$\\
\end{tabular}
\end{table}
\end{example}

\section{Properties of quasi-overlap functions}\label{sec:properties}
In this section, we first explore both the weighted invariant properties and the automorphism invariant properties of quasi-overlap functions on bounded psosets. We then introduce the migrativity, the homogeneity, the idempotency, the cancellation law, the Archimedean and limiting properties of quasi-overlap functions, and discuss their relations.
\subsection{The weighted invariant properties}

On $[0,1]$, pointwise meet and join preserve overlap functions. However, the analogous statement fails on proper trellises illustrated by the following example.
\begin{example}\label{ex:pointwise-failure}
There exist $O_1,O_2\in\mathbb{QO}(\mathbb{T}_5)$ such that neither
$$
 H_\wedge(x,y)=O_1(x,y)\wedge O_2(x,y)
$$
nor
$$
 H_\vee(x,y)=O_1(x,y)\vee O_2(x,y)
$$
is a quasi-overlap function.

\begin{proof}
Take $O_1=O_{a,a}$ and $O_2=O_{b,c}$ defined by Example \ref{examp:T5-classification}.
Since $(a,a)\trianglelefteq_2(a,1)$, the increasingness would require $H_\wedge(a,a)\trianglelefteq H_\wedge(a,1)$. However,
$$
 H_\wedge(a,a)=a\wedge b=a,
 H_\wedge(a,1)=a\wedge c=c,
$$
and $a\ntrianglelefteq c$. Similarly,
$$
 H_\vee(a,a)=a\vee b=b,
 H_\vee(a,1)=a\vee c=a,
$$
but $b\ntrianglelefteq a$. Hence neither $H_\wedge$ nor $H_\vee$ is increasing, they are not quasi-overlap functions.
\end{proof}
\end{example}

It is well-known that overlap functions are preserved on the lattice-valued setting through weighting \cite{Paiva2021}. In what follows, we also discuss the weighted invariance of quasi-overlap functions on bounded psosets.

Let $S$ be a t-conorm on a bounded psoset $\mathbb{P}=(X,\trianglelefteq,0,1)$ and $n\geq2$. Denote by $S^{(n)}$ the associative $n$-ary extension, i.e., $$S^{(n)}(x_1, x_2, \cdots, x_n)=S(S^{(n-1)}(x_1, x_2, \cdots, x_{n-1}),x_n)$$ for any $x_1, x_2, \cdots, x_n\in X$. Then the following lemma is obvious.
\begin{lemma}\label{lem:nary-boundary}
Let $S$ be a t-conorm on a bounded psoset $\mathbb{P}=(X,\trianglelefteq,0,1)$. Then, for all $x_1,\ldots,x_n\in X$,
$$
\begin{aligned}
 S^{(n)}(x_1,\ldots,x_n)=0
 &\Longleftrightarrow x_1=\cdots=x_n=0.
\end{aligned}
$$
Moreover, if $S$ is positive, then
$$
 S^{(n)}(x_1,\ldots,x_n)=1
 \Longleftrightarrow
 \mbox{ there is an }i\in\{1,\ldots,n\}\mbox{ such that }x_i=1.
$$
\end{lemma}

\begin{definition}\label{def4.6}
\emph{Let $\mathbb{P}=(X,\trianglelefteq,0,1)$ be a bounded psoset, and $S$ be a t-conorm on $\mathbb{P}$. A vector $w=(w_1,\ldots,w_n)\in X^n$ is called an $S$-\emph{weight vector} if
$$
 S^{(n)}(w_1,\ldots,w_n)=1.
$$}
\end{definition}

Notice that if $S$ is a positive t-conorm on a bounded psoset
$\mathbb{P}=(X,\trianglelefteq,0,1)$, then Lemma~\ref{lem:nary-boundary} shows that
every $S$-weight vector $w=(w_1,\ldots,w_n)$ has at least one
component that is equal to $1$.

\begin{theorem}\label{thm:weighted}
Let $\mathbb{P}=(X,\trianglelefteq,0,1)$ be a bounded psoset, $T$ be a t-norm and $S$ a positive t-conorm on $\mathbb{P}$, and $O_1,\ldots,O_n\in\mathbb{QO}(\mathbb{P})$ with $n\geq2$. If $w=(w_1,\ldots,w_n)$ is an $S$-weight vector, then the binary operation $O_w:X^2\to X$ defined by
$$
 O_w(x,y)=S^{(n)}\bigl(T(w_1,O_1(x,y)),\ldots,T(w_n,O_n(x,y))\bigr)
$$
is a quasi-overlap function on $\mathbb{P}$.
\end{theorem}

\begin{proof}
The commutativity of $O_w$ follows from that of the functions $O_i$. In the following, we first prove its increasingness. Let $(x,y)\trianglelefteq_2(x',y')$. Then
$$
 O_i(x,y)\trianglelefteq O_i(x',y')
$$ for all $i=1,\ldots,n$ since each $O_i$ is increasing.
By the increasingness of $T$,
$$
 T(w_i,O_i(x,y))\trianglelefteq T(w_i,O_i(x',y'))
 $$for all $i=1,\ldots,n$.
Hence, by the increasingness of $S^{(n)}$,
$$
 O_w(x,y)\trianglelefteq O_w(x',y').
$$
Therefore, $O_w$ satisfies (QO4).

Next, we verify (QO2). We complete it by two steps.

Step 1. If $x=0$ or $y=0$, then $O_i(x,y)=0$ for every $i=1,\ldots,n$, which implies
$$
 T(w_i,0)\trianglelefteq T(1,0)=0
$$since $w_i\trianglelefteq1$ and $T$ is increasing.
Thus $T(w_i,0)=0$. Therefore,
$$
 O_w(x,y)=S^{(n)}(0,\ldots,0)=0.
$$

Step 2. Suppose that $O_w(x,y)=0$. Then by Lemma~\ref{lem:nary-boundary},
$$
 T(w_i,O_i(x,y))=0
$$for all $i=1,\ldots,n$.
Because $w$ is an $S$-weight vector and $S$ is positive, Lemma~\ref{lem:nary-boundary} implies that there is an $i_0\in\{1,\ldots,n\}$ such that $w_{i_0}=1$. Hence
$$
 0=T(w_{i_0},O_{i_0}(x,y))
  =T(1,O_{i_0}(x,y))
  =O_{i_0}(x,y).
$$
By (QO2) for $O_{i_0}$, $x=0$ or $y=0$.

Finally, we verify (QO3). We complete it by two parts.

Part 1. Suppose that $O_w(x,y)=1$. Then from Lemma~\ref{lem:nary-boundary}, there is an $i_0\in\{1,\ldots,n\}$ such that
$$
 T(w_{i_0},O_{i_0}(x,y))=1
$$ since $S$ is positive.
Thus
$$
 w_{i_0}=1\mbox{ and }O_{i_0}(x,y)=1.
$$
By (QO3) for $O_{i_0}$, $x=y=1$.

Part 2. If $x=y=1$, then $O_i(1,1)=1$ for every $i=1,\ldots,n$, and
$$
\begin{aligned}
O_w(1,1)
&=S^{(n)}\bigl(T(w_1,1),\ldots,T(w_n,1)\bigr)\\
&=S^{(n)}(w_1,\ldots,w_n)\\
&=1
\end{aligned}
$$ since $w$ is an $S$-weight vector.

Therefore, from Definition~\ref{def:qo}, $O_w\in\mathbb{QO}(\mathbb{P})$.
\end{proof}

\subsection{Automorphic transforms}

Similarly to the theory of lattices, one easily gives the definition of an automorphism mapping on a bounded psoset:
Let $\mathbb{P}=(X,\trianglelefteq,0,1)$ be a bounded psoset. A bijection $\rho:X\to X$ is called an \emph{automorphism} of $\mathbb{P}$ if, for all $x,y\in X$,
$$
x\trianglelefteq y
\Longleftrightarrow
\rho(x)\trianglelefteq\rho(y).
$$
Since $\mathbb{P}$ is bounded, every automorphism necessarily fixes
the least and greatest elements, i.e., $\rho(0)=0$ and $\rho(1)=1$.

The following proposition shows that a quasi-overlap function is preserved under an automorphism.
\begin{proposition}\label{prop:automorphism}
Let $\mathbb{P}=(X,\trianglelefteq,0,1)$ be a bounded psoset, $\rho$ be an automorphism on $\mathbb{P}$ and $O\in\mathbb{QO}(\mathbb{P})$. Then the binary operation $O^\rho:X^2\to X$ defined by
\begin{equation}\label{eq1}
 O^\rho(x,y)=\rho^{-1}\bigl(O(\rho(x),\rho(y))\bigr)
\end{equation}
is a quasi-overlap function on $\mathbb{P}$.
\end{proposition}

\begin{proof}
Condition (QO1) is immediate. Since $\rho$ fixes $0$ and $1$, the conditions (QO2) and (QO3) are preserved. It remains to verify (QO4). Let $(x,y)\trianglelefteq_2(x',y')$. Then
$$
 (\rho(x),\rho(y))\trianglelefteq_2(\rho(x'),\rho(y'))
$$ since $\rho$ is order preserving.
The increasingness of $O$ yields
$$
 O(\rho(x),\rho(y))\trianglelefteq O(\rho(x'),\rho(y')).
$$
Because $\rho^{-1}$ is also order preserving, we obtain
$$
 O^\rho(x,y)\trianglelefteq O^\rho(x',y').
$$
Thus $O^\rho\in\mathbb{QO}(\mathbb{P})$.
\end{proof}

\begin{corollary}\label{cor:tnorm-generated}
Let $\mathbb{P}=(X,\trianglelefteq,0,1)$ be a bounded psoset and $T$ be a positive t-norm on $\mathbb{P}$. Let $\rho$ be an automorphism on $\mathbb{P}$, and $f:X\to X$ be an increasing mapping such that
$
 f(t)=0\Longleftrightarrow t=0\mbox{ and }
 f(t)=1\Longleftrightarrow t=1.
$
Then
$$
 O_{T,\rho,f}(x,y)=f\!\left(\rho^{-1}(T(\rho(x),\rho(y)))\right)
$$
is a quasi-overlap function on $\mathbb{P}$.
\end{corollary}

\begin{proof}
Because the positive t-norm $T$ is a quasi-overlap function on $\mathbb{P}$, it follows from Proposition \ref{prop:automorphism} that $T^\rho$ satisfies the conditions (QO2) and (QO3). Therefore, from the properties of $f$, the composition
$$
 O_{T,\rho,f}=f\circ T^\rho
$$
is commutative and increasing and satisfies (QO2) and (QO3), i.e., $O_{T,\rho,f}$ is a quasi-overlap function on $\mathbb{P}$.
\end{proof}

\begin{example}\label{prop:T5-orbits}
\emph{In Example \ref{examp:T5-classification}, we have that $$\mathbb{QO}(\mathbb{T}_5)=\{O_{\eta,\theta}:\eta\in C,\ \theta\in\{\eta,\sigma(\eta)\}\}=\{O_{a,a},O_{b,b},O_{c,c},O_{a,b},O_{b,c},O_{c,a}\}.$$
There are three automorphisms on $\mathbb{T}_5$, denote them by $\rho_1, \rho_2$ and $\rho_3$, as follows:
$$\rho_1: 0\rightarrow 0, a\rightarrow a, b\rightarrow b, c\rightarrow c, 1\rightarrow 1,$$
  $$\rho_2: 0\rightarrow 0, a\rightarrow b, b\rightarrow c, c\rightarrow a, 1\rightarrow 1,$$
 $$\rho_3: 0\rightarrow 0, a\rightarrow c, b\rightarrow a, c\rightarrow b, 1\rightarrow 1.$$
Applying Proposition \ref{prop:automorphism}, one can easily check that $$O_{a,a}{\underset{\rho_3}{\overset{\rho_2}{\rightleftarrows}}} O_{c,c}{\underset{\rho_3}{\overset{\rho_2}{\rightleftarrows}}} O_{b,b}, O_{a,b}{\underset{\rho_3}{\overset{\rho_2}{\rightleftarrows}}} O_{c,a}{\underset{\rho_3}{\overset{\rho_2}{\rightleftarrows}}} O_{b,c}.$$
Therefore, from the viewpoint of automorphism, $\mathbb{QO}(\mathbb{T}_5)$ can be exactly divided two classes as below:
 $$\{O_{a,a},O_{b,b},O_{c,c}\} \mbox{ and }
 \{O_{a,b},O_{b,c},O_{c,a}\}.
$$
In a completely analogous way, $\mathbb{QG}(\mathbb{T}_5)$ can be exactly divided two classes as below:
 $$\{G_{a,a},G_{b,b},G_{c,c}\} \mbox{ and }
 \{G_{a,b},G_{b,c},G_{c,a}\}.
$$}
\end{example}
\subsection{Migrativity}
We apply the generalized notion of migrativity introduced in \cite{Qiao2021} to binary operations on psosets as follows.
\begin{definition}[\cite{Qiao2021}]
\emph{Let $A,B,C:X^2\to X$ be three binary operations on a bounded psoset $\mathbb{P}=(X,\trianglelefteq,0,1)$ and $\alpha\in X$. The operation $A$ is said to be $(\alpha,B,C)$-\emph{migrative} if, for all $x,y\in X$,
$$
 A(B(\alpha,x),y)=A(x,C(\alpha,y)).
$$
Moreover, $A$ is said to be $(B,C)$-migrative if it is $(\alpha,B,C)$-\emph{migrative} for all $\alpha\in X$, in particular, it is said to be $B$-migrative if $B=C$.}
\end{definition}

Then we have the following statement which describes the $B$-migrativity of a binary operation.
\begin{proposition}\label{prop:migrative-factorization}
Let $\mathbb{P}=(X,\trianglelefteq,0,1)$ be a bounded psoset, $A:X^2\to X$ be a commutative binary operation and $B:X^2\to X$ be a commutative and associative binary operation with neutral element $e$. Then the following statements are equivalent:
\begin{enumerate}[label=(\roman*)]
\item $A$ is $B$-migrative;
\item $A(x,y)=A(e,B(x,y))$ for all $x,y\in X$;
\item There exists a unary operation $f:X\to X$ such that $A(x,y)=f(B(x,y))$ for all $x,y\in X$. In particular, one may take $f(x)=A(e,x)$ for all $x\in X$.
\end{enumerate}
\end{proposition}

\begin{proof}
Taking $\alpha=y$, $x=e$ and $y=x$, it follows from the commutativity of both $A$ and $B$ that (i) implies (ii).

If (ii) holds, then defining $f(x)=A(e,x)$ for all $x\in X$ leads to
$$
 A(x,y)=f(B(x,y)),
$$
and hence (iii) follows.

Finally, suppose that (iii) holds. Then for all $\alpha,x,y\in X$, the associativity and commutativity of $B$ yield
$$
\begin{aligned}
A(B(\alpha,x),y)
&=f(B(B(\alpha,x),y))\\
&=f(B(x,B(\alpha,y)))\\
&=A(x,B(\alpha,y)).
\end{aligned}
$$
Thus $A$ is $B$-migrative.
\end{proof}

In general, the associativity of $B$ in Proposition \ref{prop:migrative-factorization} cannot be deleted as shown by the following example.
\begin{example}\label{ex:migrativity-assoc-essential}
\upshape
On $\mathbb{T}_5$ in Example~\ref{ex:T5}, let $A=B=\wedge$. Clearly, $B$ is not associative with  neutral element $1$. Then $A=\operatorname{id}\circ B$, so Proposition~\ref{prop:migrative-factorization} (iii) holds. We have
$$
 A(B(b,a),c)=(b\wedge a)\wedge c=a\wedge c=c,
$$
whereas
$$
 A(a,B(b,c))=a\wedge(b\wedge c)=a\wedge b=a.
$$
Thus $A$ is not $B$-migrative.
\end{example}

In particular, the following theorem characterizes a $T$-migrative quasi-overlap function, where $T$ is a t-norm on a bounded psoset.
\begin{theorem}\label{thm:tnorm-migrative}
Let $T:X^2\to X$ be a t-norm and $O:X^2\to X$ a binary operation on a bounded psoset $\mathbb{P}=(X,\trianglelefteq,0,1)$. The following statements are equivalent:
\begin{enumerate}[label=(\roman*)]
\item $O$ is a $T$-migrative quasi-overlap function;
\item $T$ is positive and there exists an increasing mapping $f:X\to X$ such that
$f(x)=0\Longleftrightarrow x=0,
f(x)=1\Longleftrightarrow x=1$
and
$O(x,y)=f(T(x,y))$ for all $x,y\in X$.
Moreover, the  mapping $f:X\to X$ is unique and is given by $f(x)=O(1,x)$.
\end{enumerate}
\end{theorem}

\begin{proof}
$(i)\Rightarrow(ii)$.
Assume that $O$ is a $T$-migrative quasi-overlap function. Since $T$
is commutative, associative, and has the neutral element $1$,
 according to Proposition~\ref{prop:migrative-factorization},
$$O(x,y)=f(T(x,y))\mbox{ with }f(x)=O(1,x)\mbox{ for all $x,y\in X$}.$$
As $O$ is increasing, so is $f$. Moreover, by (QO2) and (QO3),
$f(x)=0\Longleftrightarrow x=0, f(x)=1\Longleftrightarrow x=1.$

It remains to verify that $T$ is positive. If $T(x,y)=0$,
$$
 O(x,y)=f(0)=0,
$$
by (QO2), $x=0$ or $y=0$, thus $T$ is positive.

The uniqueness of $f$ is obvious.

$(ii)\Rightarrow(i)$. Suppose that $T$ is positive and $f:X\to X$ is increasing, and they satisfy
$f(x)=0\Longleftrightarrow x=0,
 f(x)=1\Longleftrightarrow x=1$
and
$$
 O(x,y)=f(T(x,y))\mbox{ for all $x,y\in X$}.
$$
Since both $T$ and $f$ are increasing and $T$ is commutative, $O$ is
increasing and commutative. Furthermore,
$$
\begin{aligned}
 O(x,y)=0
 &\Longleftrightarrow T(x,y)=0\\
 &\Longleftrightarrow x=0\mbox{ or }y=0,
\end{aligned}
$$
so (QO2) holds. Meanwhile,
$$
\begin{aligned}
 O(x,y)=1
 &\Longleftrightarrow T(x,y)=1\\
 &\Longleftrightarrow x=y=1.
 \end{aligned}
$$
Hence $O$ is a quasi-overlap function. Therefore, $O$ is $T$-migrative directly from Proposition~\ref{prop:migrative-factorization}.
\end{proof}

\subsection{Homogeneity}
We apply the generalized notion of homogeneity introduced in \cite{Qiao2021} to binary operations on psosets as follows.
\begin{definition}[\cite{Qiao2021}]\label{def:homogeneity}
\emph{Let $A,B,C:X^2\to X$ be three binary operations on a  bounded psoset $\mathbb{P}=(X,\trianglelefteq,0,1)$. Then $A$ is said to be $(B,C)$-\emph{homogeneous} if, for all $\lambda,x,y\in X$, it satisfies}
$$
 A(B(\lambda,x),B(\lambda,y))=C(\lambda,A(x,y)).
$$
\end{definition}

Then we first have the following proposition.
\begin{proposition}\label{prop:homogeneity-basic}
Let $A,B,C:X^2\to X$ be three binary operations on a bounded psoset $\mathbb{P}=(X,\trianglelefteq,0,1)$ where $B$ has $e$ as its neutral element. Let $A$ be $(B,C)$-homogeneous. Then
\begin{enumerate}[label=(\roman*)]
    \item $A(x,x)=C(x,A(e,e))$ for any $x\in X$;
    \item $A(x,y)=C(e,A(x,y))$ for any $x\in X$.
\end{enumerate}
\end{proposition}

\begin{proof}
It can be obtained immediately from Definition~\ref{def:homogeneity}.
\end{proof}

\begin{corollary}\label{cor:homogeneity-diagonal}
Let $\mathbb{T}=(X,\trianglelefteq,\wedge,\vee,0,1)$ be a bounded trellis, $O\in\mathbb{QO}(\mathbb{T})$. Let $B,C:X^2\to X$ be two binary operations on $\mathbb{T}$ where $B$ has $1$ as its neutral element. If $O$ is $(B,C)$-homogeneous, then
$$
 O(x,x)=C(x,1)
$$
for every $x\in X$.
\end{corollary}

Next, we give three sufficient conditions for quasi-overlap functions to be $(B,C)$-homogeneous.
\begin{proposition}
Let $\mathbb{P}=(X,\trianglelefteq,0,1)$ be a bounded psoset, $T$ be a t-norm and $S$ a positive t-conorm on $\mathbb{P}$. Let $O_1,\ldots,O_n\in\mathbb{QO}(\mathbb{P})$ with $n\geq2$. If every $O_i$ be $(B,C)$-homogeneous, $w=(w_1,\ldots,w_n)$ be an $S$-weight vector, and for all $\lambda,z,z_1,\ldots,z_n\in X$ the following two equalities hold:
\begin{enumerate}[label=(\roman*)]
    \item $T(w_i,C(\lambda,z))=C(\lambda,T(w_i,z))\mbox{ for every }i=1,\ldots,n$;
\item $\begin{aligned}
&S^{(n)}(C(\lambda,z_1),\ldots,C(\lambda,z_n))
=C(\lambda,S^{(n)}(z_1,\ldots,z_n)).\end{aligned}$
\end{enumerate}
Then $O_w$ is $(B,C)$-homogeneous.
\end{proposition}

\begin{proof}
Note that for all $\lambda,x,y\in X$,
$$
 O_i(B(\lambda,x),B(\lambda,y))
 =C(\lambda,O_i(x,y))
 $$ for every $i=1,\ldots,n$ since each $O_i$ is $(B,C)$-homogeneous.
Thus from Theorem \ref{thm:weighted},
$$
\begin{aligned}
&O_w(B(\lambda,x),B(\lambda,y))\\
&=S^{(n)}\bigl(
T(w_1,C(\lambda,O_1(x,y))),\ldots,
T(w_n,C(\lambda,O_n(x,y)))
\bigr)\\
&=S^{(n)}\bigl(
C(\lambda,T(w_1,O_1(x,y))),\ldots,
C(\lambda,T(w_n,O_n(x,y)))
\bigr)\\
&=C\!\left(
\lambda,
S^{(n)}\bigl(
T(w_1,O_1(x,y)),\ldots,
T(w_n,O_n(x,y))
\bigr)
\right)\\
&=C(\lambda,O_w(x,y)).
\end{aligned}
$$
Therefore, $O_w$ is $(B,C)$-homogeneous.
\end{proof}

\begin{proposition}\label{prop:homogeneity-transport}
Let $\mathbb{P}=(X,\trianglelefteq,0,1)$ be a bounded psoset, $O\in\mathbb{QO}(\mathbb{P})$ and $\rho$ be an automorphism on $\mathbb{P}$. If $O$ is $(B,C)$-homogeneous, then $O^\rho$ defined by Eq.(\ref{eq1}) is $(B^\rho,C^\rho)$-homogeneous.
\end{proposition}

\begin{proof}
Note that for all $\lambda,x,y\in X$,
$$
\begin{aligned}
O^\rho(B^\rho(\lambda,x),B^\rho(\lambda,y))
&=\rho^{-1}\!\left(
 O(\rho(B^{\rho}(\lambda,x)),\rho(B^{\rho}(\lambda,y)))
 \right)\\
&=\rho^{-1}\!\left(
 O(\rho\circ\rho^{-1}(B(\rho(\lambda),\rho(x))),\rho\circ\rho^{-1}(B(\rho(\lambda),\rho(y))))
 \right)\\
 &=\rho^{-1}\!\left(
 O(B(\rho(\lambda),\rho(x)),B(\rho(\lambda),\rho(y)))
 \right)\\
&=\rho^{-1}\!\left(
 C(\rho(\lambda),O(\rho(x),\rho(y)))
 \right)\\
&=\rho^{-1}\!\left(
 C(\rho(\lambda),\rho\circ\rho^{-1}(O(\rho(x),\rho(y))))
 \right)\\
 &=\rho^{-1}\!\left(
 C(\rho(\lambda),\rho(O^{\rho}(x,y)))
 \right)\\
&=C^\rho(\lambda,O^\rho(x,y)).
\end{aligned}$$
Hence $O^\rho$ is $(B^\rho,C^\rho)$-homogeneous.
\end{proof}

Let $F:X^2\to X$ be a binary operation with the neutral element $1$ on a bounded psoset $\mathbb{P}=(X,\trianglelefteq,0,1)$ and define
$$\lambda_F^{(0)}=1, \lambda_F^{(1)}=\lambda, \lambda_F^{(n)}=F(\lambda_F^{(n-1)},\lambda) \mbox{ with }n\ge2.$$
Then similar to \cite{Qiao2021}, we can give the definition of an $F^k$-homogeneous binary operation: a binary operation $A:X^2\to X$ is said to be $F^k$-\emph{homogeneous} if, for all $\lambda,x,y\in X$,
$$
 A(F(\lambda,x),F(\lambda,y))
 =F(\lambda_F^{(k)},A(x,y)).
$$

\begin{proposition}\label{prop:tnorm-homogeneous}
Let $\mathbb{P}=(X,\trianglelefteq,0,1)$ be a bounded psoset and let $T$ be a positive t-norm on $\mathbb{P}$. Then $T$ is a quasi-overlap function and is $T^2$-homogeneous.
\end{proposition}
\begin{proof}
It is easy to check that $T$ is a quasi-overlap function.
Using the associativity and commutativity of $T$, we obtain
$$
\begin{aligned}
T(T(\lambda,x),T(\lambda,y))
&=T(\lambda,T(x,T(\lambda,y)))\\
&=T(\lambda,T(\lambda,T(x,y)))\\
&=T(T(\lambda,\lambda),T(x,y)).
\end{aligned}
$$
Therefore, $T$ is $T^2$-homogeneous.
\end{proof}

\subsection{Idempotency and cancellation law}

In an analogous way to \cite{Klement2000}, we may give the following definition on a bounded psoset.
\begin{definition}
\emph{Let $\mathbb{P}=(X,\trianglelefteq,0,1)$ be a bounded psoset and $O\in\mathbb{QO}(\mathbb{P})$. An element $x\in X$ is called an idempotent element of $O$ if $O(x,x)=x$. If every $x\in X$ is an idempotent element of $O$, then $O$ is called \emph{idempotent}. The numbers $0$ and $1$ (which are idempotent elements of any quasi-overlap function $O$) are called trivial idempotent elements of $O$, each idempotent element in $X\setminus \{0,1\}$ will be called a non-trivial idempotent element of $O$.}
\end{definition}

No quasi-overlap function on $\mathbb{T}_5$ is idempotent. Indeed,
let $O=O_{\eta,\theta}\in\mathbb{QO}(\mathbb{T}_5)$. By Example~\ref{examp:T5-classification},
$ O(x,x)=\eta \mbox{ for every }x\in C.$
If $O$ is idempotent, then $\eta=x$ for every $x\in C$, which is impossible.

The following proposition gives a sufficient condition that a bounded trellis is lattice by an idempotent quasi-overlap function, which can be seen as a generalization of Proposition 4.2 of \cite{Zedam2023} when the t-norm is positive.
\begin{proposition}\label{prop:idempotent-forces-lattice}
Let $\mathbb{T}=(X,\trianglelefteq,\wedge,\vee,0,1)$ be a bounded trellis. If $O\in\mathbb{QO}(\mathbb{T})$ is idempotent and has the neutral element $1$, then $O=\wedge$ and $\mathbb{T}$ is a bounded lattice.
\end{proposition}

\begin{proof}
We first show that $O$ coincides with the meet operation, without assuming associativity. Let $x,y\in X$. Since $1$ is the neutral element of $O$,  the increasingness gives
$
O(x,y)\trianglelefteq O(x,1)=x,
O(x,y)\trianglelefteq O(1,y)=y.
$
Thus $O(x,y)$ is a common lower bound of $x$ and $y$. Conversely, if $u\in X$ satisfies $u\trianglelefteq x$ and $u\trianglelefteq y$, then idempotency and increasingness with respect to the product pseudo-order yield
$
u=O(u,u)\trianglelefteq O(x,y).
$
Therefore, $O(x,y)$ is the greatest lower bound of $x$ and $y$, and hence $O(x,y)=x\wedge y$.

It remains to prove that $\trianglelefteq$ is transitive. Let $x,y,z\in X$ satisfy $x\trianglelefteq y$ and $y\trianglelefteq z$. Since $x\trianglelefteq y$, we have $x\wedge y=x$. The increasingness of $O=\wedge$ implies
$
x=x\wedge y\trianglelefteq x\wedge z.
$
On the other hand, $x\wedge z\trianglelefteq x$ by the definition of meet. Antisymmetry therefore gives $x=x\wedge z$. Since $x\wedge z\trianglelefteq z$, it follows that $x\trianglelefteq z$.
Thus $\trianglelefteq$ is transitive, so $\mathbb{T}$ is a bounded lattice.
\end{proof}

\begin{definition}
\emph{Let $O$ be a quasi-overlap function on a bounded psoset $\mathbb{P}=(X,\trianglelefteq,0,1)$. $O$ is said to satisfy the \emph{cancellation law} if, for all $x,y,z\in X$,
$$
 O(x,y)=O(x,z)\Longrightarrow x=0\mbox{ or }y=z.
$$
 $O$ is said to be the \emph{strictly increasing} if, for all $x,y,z\in X$,
$$
 x\neq0, y\triangleleft z\Longrightarrow O(x,y)\triangleleft O(x,z).
$$}
\end{definition}

\begin{proposition}\label{prop:cancellation-injective}
Let $\mathbb{P}=(X,\trianglelefteq,0,1)$ be a bounded psoset and $O\in\mathbb{QO}(\mathbb{P})$. The following statements are equivalent:
\begin{enumerate}[label=(\roman*)]
\item $O$ satisfies the cancellation law;
\item For every $x\neq0$, the mapping $O(x,\cdot):X\to X$ is injective.
\end{enumerate}
Moreover, either condition implies that $O$ is strictly increasing.
\end{proposition}

\begin{proof}
The equivalence between (i) and (ii) follows directly from the definition of the cancellation law. Assume that these conditions hold, and let $x\neq0$ and $y\triangleleft z$. Since $O$ is increasing,
$
 O(x,y)\trianglelefteq O(x,z).
$
If the equality holds, then the cancellation law implies $y=z$, contrary to $y\triangleleft z$. Hence
$
 O(x,y)\triangleleft O(x,z),
$
and $O$ is strictly increasing.
\end{proof}

\begin{proposition}\label{prop:finite-no-cancellation}
Let $\mathbb{P}=(X,\trianglelefteq,0,1)$ be a finite bounded psoset with $|X|>2$. Then no quasi-overlap function on $\mathbb{P}$ satisfies the cancellation law.
\end{proposition}

\begin{proof}
Suppose, to the contrary, that $O\in\mathbb{QO}(\mathbb{P})$ satisfies the cancellation law. Choose $x\in X\setminus\{0,1\}$. By Proposition~\ref{prop:cancellation-injective}, the mapping
$
 O(x,\cdot):X\to X
$
is injective. Since $X$ is finite, it is also surjective. Hence there exists $y\in X$ such that $O(x,y)=1.$ Condition (QO3) gives $x=y=1$, contradicting the choice of $x$. Therefore, no quasi-overlap function on $\mathbb{P}$ is cancellative.
\end{proof}

\begin{proposition}\label{prop:complete-strict-cancel}
Let $\mathbb{P}=(X,\trianglelefteq,0,1)$ be a bounded psoset and assume that every two distinct elements are comparable with respect to $\trianglelefteq$. Then a quasi-overlap function on the bounded psoset is cancellative if and only if it is strictly increasing.
\end{proposition}

\begin{proof}
The necessity follows from Proposition~\ref{prop:cancellation-injective}.

Now, suppose that $O: X^2\to X$ is a strictly increasing quasi-overlap function and
$$
 O(x,y)=O(x,z)\mbox{ with }x\neq 0.
$$
 If $y\neq z$, then either $y\triangleleft z$ or $z\triangleleft y$. In the first case,
$$
 O(x,y)\triangleleft O(x,z),
$$
and in the second case,
$$
 O(x,z)\triangleleft O(x,y).
$$
Both cases contradict $O(x,y)=O(x,z)$. Hence $y=z$, and therefore $O$ satisfies the cancellation law.
\end{proof}

Applying Propositions \ref{prop:finite-no-cancellation} and \ref{prop:complete-strict-cancel}, we have the following corollary.
\begin{corollary}
No quasi-overlap function on $\mathbb{T}_5$ is strictly increasing.
\end{corollary}

The following example shows that the condition ``every two distinct elements are comparable with respect to $\trianglelefteq$" in Proposition \ref{prop:complete-strict-cancel} cannot be deleted, in general.
\begin{example}\label{ex:strict-not-cancel}
\upshape
Let $L=[0,1]^2$ be equipped with the coordinatewise order, and define $h: [0,1]^2\to [0,1]$ by $h(u_1,u_2)=\frac{u_1+u_2}{2}$ and $d:[0,1]\to [0,1]^2$ by $d(t)=(t,t)$. Set $O: L^2\to L$ by $O(u,v)=d(h(u)h(v))$. Then $O$ is a quasi-overlap function on $L$. Indeed, $h(u)=0$ if and only if $u=(0,0)$, $h(u)=1$ if and only if $u=(1,1)$, and $h$ is increasing.

If $u\neq(0,0)$ and $v<w$ in the coordinatewise order, then
$
 h(u)>0
 \mbox{ and }
 h(v)<h(w),
$
so
$
 O(u,v)<O(u,w).
$
Thus $O$ is strictly increasing. However, for
$
 e=(1/4,3/4),
 f=(3/4,1/4),
$
we have $e\parallel f$ and $h(e)=h(f)=1/2$. Consequently,
$
 O((1,1),e)=O((1,1),f),
$
so the cancellation law fails.
\end{example}

Therefore, on a bounded trellis, a strictly increasing quasi-overlap function does not imply the cancellation law, generally.

\subsection{Archimedean and limiting properties}

Motivated by the lattice-valued setting in \cite{Paiva2021}, we formulate Archimedean and limiting properties of quasi-overlap functions directly on a bounded psoset.
\begin{definition}\label{def:archimedean}
\emph{A quasi-overlap function $O$ on a bounded psoset $\mathbb{P}=(X,\trianglelefteq,0,1)$ is called \emph{Archimedean} if, for any $x,y\in X\setminus\{0,1\}$, there exists an $n\ge1$ such that
$x_O^{(n)}\triangleleft y.$}
\end{definition}

\begin{definition}\label{def:order-limiting}
\emph{A quasi-overlap function $O$ on a bounded psoset $\mathbb{P}=(X,\trianglelefteq,0,1)$ is said to have the \emph{limiting property} if, for any $x\in X\setminus\{0,1\}$ and $y\in X\setminus\{0\}$, there exists an $N\ge2$ such that
$x_O^{(n)}\triangleleft y
\mbox{ for all }n\ge N.$}
\end{definition}

Then we have the following three propositions which provide the relationship among the Archimedean, the limiting property and the idempotency.
\begin{proposition}\label{prop:limiting-implies-arch}
Every quasi-overlap function satisfying the limiting property on a bounded psoset $\mathbb{P}=(X,\trianglelefteq,0,1)$ is Archimedean.
\end{proposition}

\begin{proof}
Let $x,y\in X\setminus\{0,1\}$. Since $0$ is the smallest element and $y\neq0$, we have $0\triangleleft y$. By the limiting property, there exists an $N\ge2$ such that
$
x_O^{(n)}\triangleleft y
\mbox{ for all }n\ge N.
$
In particular, $x_O^{(N)}\triangleleft y.$ Hence $O$ is Archimedean.
\end{proof}
\begin{proposition}\label{prop:arch-limiting}
Let $O$ be an Archimedean quasi-overlap function on a bounded psoset
$\mathbb{P}=(X,\trianglelefteq,0,1)$. If $\mathbb{P}=(X,\trianglelefteq,0,1)$ is a totally ordered set, then $O$ has the limiting property.
\end{proposition}

\begin{proof}
Fix $x\in X\setminus\{0,1\}$. We first show that
\begin{equation}\label{eq9}x_O^{(2)}\triangleleft x.
\end{equation}
In fact, by taking $y=x$ there exists an $m\ge1$ such that
$$
x_O^{(m)}\triangleleft x
$$
since $O$ is Archimedean. Clearly, $m\ge2$. Suppose that $x_O^{(2)}\ntriangleleft x$. Since $\trianglelefteq$ is a total order, we have
$$
x=x_O^{(1)}\trianglelefteq x_O^{(2)}.
$$
By the increasingness of $O$, it follows inductively that
$$
x_O^{(1)}\trianglelefteq x_O^{(2)}
\trianglelefteq x_O^{(3)}
\trianglelefteq\cdots.
$$
In particular,
$$
x\trianglelefteq x_O^{(m)},
$$
contrary to $x_O^{(m)}\triangleleft x$.

We next prove that
\begin{equation}\label{eq10}
x_O^{(n+1)}\trianglelefteq x_O^{(n)}\mbox{ when }n\ge1.
\end{equation}
Indeed, by Eq.\eqref{eq9}, $x_O^{(2)}\triangleleft x$. If
$x_O^{(n)}\trianglelefteq x_O^{(n-1)}$, then, by the increasingness of $O$,
$$
x_O^{(n+1)}
=O(x_O^{(n)},x)
\trianglelefteq
O(x_O^{(n-1)},x)
=x_O^{(n)}.
$$
Thus Eq.\eqref{eq10} holds by induction.

Now, let $y\in X\setminus \{0\}$. We distinguish two cases as follows.

Case 1. if $y\neq1$, then
$y\in X\setminus\{0,1\}$. By the Archimedean property, there exists
$k\ge1$ such that
$$
x_O^{(k)}\triangleleft y.
$$
Put $N=\max\{2,k\}$. Then by Eq.\eqref{eq10}, for every
$n\ge N$,
$$
x_O^{(n)}
\trianglelefteq x_O^{(k)}
\triangleleft y,
$$
and hence, by transitivity,
$$
x_O^{(n)}\triangleleft y.
$$

Case 2. if $y=1$, then by Eqs. \eqref{eq9} and \eqref{eq10}, for any $n\ge1$ we have
$$
x_O^{(n)}\triangleleft1=y.
$$

Cases 1 and 2 imply that $O$ has the
limiting property.
\end{proof}

The following example shows that the condition that $\mathbb{P}=(X,\trianglelefteq,0,1)$ is a totally ordered set in Proposition~\ref{prop:arch-limiting} cannot be dropped, in general.
\begin{example}\label{ex:arch-not-limiting}
\upshape
Let $X=[0,1]\cup\{a\}$ with $a\notin[0,1]$, and put
$$
A=(0,1/2]\setminus\{2^{-2^m}:m\ge1\}.
$$
Besides reflexivity and the usual order on $[0,1]$, define
$
0\triangleleft a\triangleleft1
$
and, for $x\in(0,1)$,
$
x\triangleleft a\Longleftrightarrow x\in A,
$
while $a\parallel x$ for every $x\in(0,1)\setminus A$. Then one can easily check that
$\mathbb{P}=(X,\trianglelefteq,0,1)$ is a bounded psoset. Notice that
$
2^{-4}\triangleleft2^{-3}\triangleleft a
\mbox{ but }
2^{-4}\parallel a.
$
Thus $\mathbb{P}$ is not a totally ordered set.

Define $O:X^2\to X$ by
$
O(x,y)=r(x)r(y),
$
where $r:X\to[0,1]$ is given by
$
r(x)=x\mbox{ if }x\in[0,1],
r(a)=1/2.
$
Since $x\triangleleft a$ implies $x\le1/2=r(a)$, the mapping $r$ is increasing. Hence one can easily show that
$O\in\mathbb{QO}(\mathbb{P})$.

We now show that $O$ is Archimedean. Let
$x\in X\setminus\{0,1\}$ and put
$
q=r(x)\in(0,1).
$
A straightforward induction gives
$
x_O^{(n)}=q^n\mbox{ when }n\ge2.
$
Let $y\in X\setminus\{0,1\}$. If $y\in(0,1)$, then there exists
$n\ge2$ such that
$
q^n<y,
$
and hence
$
x_O^{(n)}\triangleleft y.
$
Now, suppose that $y=a$. Choose $N\ge2$ such that
$
q^N\le1/2.
$
We claim that one of $q^N$ and $q^{N+1}$ belongs to $A$. Indeed, write
$q=2^{-s}$ for some $s>0$. If
$
q^N=2^{-2^i}
\mbox{ and }
q^{N+1}=2^{-2^j}
$
for some $i,j\ge1$, then
$
Ns=2^i, (N+1)s=2^j.
$
Since $q<1$, we have $j>i$, and hence
$$
\frac{N+1}{N}=2^{j-i}\ge2,
$$
which is impossible for $N\ge2$. Thus either
$q^N\in A$ or $q^{N+1}\in A$, and consequently
$
x_O^{(n)}\triangleleft a
$
for some $n\ge2$. Therefore, $O$ is Archimedean.

However, $O$ does not satisfy the limiting property. Indeed, take
$x=y=a$. For every $m\ge1$,
$$
a_O^{(2^m)}=2^{-2^m}\notin A,
$$
and hence
$
a_O^{(2^m)}\parallel a.
$
Since $\lim_{m\to\infty}2^m=\infty$, there is no $N\ge2$ such that
$
a_O^{(n)}\triangleleft a
$
for all $n\ge N$. Thus $O$ is Archimedean but does not have the limiting property.
\end{example}

\begin{proposition}\label{prop:arch-idempotent}
Let $O$ be an Archimedean quasi-overlap function on a bounded psoset $\mathbb{P}=(X,\trianglelefteq,0,1)$, then $O$ has only trivial idempotent elements.
\end{proposition}

\begin{proof}
Suppose that $a\in X\setminus\{0,1\}$ is an idempotent element of $O$, i.e., $O(a,a)=a.$ By induction, $a_O^{(n)}=a$ when $n\ge1$.
Applying the Archimedean property with $x=y=a$, there exists an $n\ge1$ such that
$
 a_O^{(n)}\triangleleft a.
$
Hence $a\triangleleft a$, which is impossible. Therefore, $O$ has no non-trivial idempotent elements.
\end{proof}

\begin{example}\label{cor:T5-not-arch}
\emph{(i) The function $O:[0,1]^2\to [0,1]$ with $O(x,y)=\min\{x,y\}$ is not Archimedean.}

\emph{(ii) No quasi-overlap function on $\mathbb{T}_5$ is Archimedean. Consequently, no quasi-overlap function on $\mathbb{T}_5$ has the limiting property. Indeed, let $O=O_{\eta,\theta}\in\mathbb{QO}(\mathbb{T}_5)$. By Example~\ref{examp:T5-classification},
$
 O(\eta,\eta)=\eta,
$
where $\eta\in C\subseteq X\setminus\{0,1\}$. Thus $O$ has a non-trivial idempotent element. By Proposition~\ref{prop:arch-idempotent}, $O$ is not Archimedean. Therefore, $O$ does not satisfy the limiting property from Proposition~\ref{prop:limiting-implies-arch}.}
\end{example}

Furthermore, we have the following result which gives a finite obstruction to the limiting property.
\begin{proposition}\label{prop:finite-no-limiting}
Let $\mathbb{P}=(X,\trianglelefteq,0,1)$ be a finite bounded psoset with $|X|>2$. Then no quasi-overlap function on $\mathbb{P}$ has the limiting property.
\end{proposition}

\begin{proof}
Suppose, to the contrary, that $O\in\mathbb{QO}(\mathbb{P})$ has the limiting property. Choose $x\in X\setminus\{0,1\}$. Since $x\neq0$, the condition~(QO2), by induction, implies that
\begin{equation}\label{eq6}
x_O^{(n)}\neq0\mbox{ for every }n\ge1.
\end{equation}
Then for each $y\in X\setminus\{0\}$, the limiting property provides an integer $N_y\ge2$ such that
\begin{equation}\label{eq7}
x_O^{(n)}\triangleleft y
\mbox{ for all }n\ge N_y.
\end{equation}
Since $X\setminus\{0\}$ is finite and nonempty, the integer
$$
N=\max\{N_y:y\in X\setminus\{0\}\}
$$
exists. Set $z=x_O^{(N)}$. Then by Eq.\eqref{eq6}, $z\neq0$, so $N_z\le N$. In Eq.\eqref{eq7}, substituting $z$ for $y$ and $N$ for $n$, respectively, yields
$$
z=x_O^{(N)}\triangleleft z,
$$
a contradiction. Therefore, no quasi-overlap function on $\mathbb{P}$ has the limiting property.
\end{proof}

We finally give a quasi-overlap function with the limiting property on an infinite bounded proper trellis.

\begin{example}\label{ex:infinite-archimedean}
\upshape
Let $X=[0,1]\cup\{a,b,c\}$ with $a,b,c\notin[0,1]$. Define $\trianglelefteq$ on $X$ by distinguishing three cases as follows.

(i) For every $x,y\in[0,1]$, $x \trianglelefteq y$ if and only if $x\leq y$.

(ii) For every $\eta\in\{a,b,c\}$, define
$x\triangleleft \eta$ if $x\in[0,1/2]$, and $\eta\triangleleft x$ if $x\in(1/2,1]$.

(iii) On $\{a,b,c\}$, $a\trianglelefteq b\trianglelefteq c\trianglelefteq a$.\\
Then one can easily check that $(X, \trianglelefteq, 0,1)$ is a bounded psoset. Meanwhile, every pair of $X$ has a meet and a join. Indeed, pairs in $[0,1]$ use the usual minimum and maximum. For any $x\in[0,1]$ and $\eta\in\{a,b,c\}$, the elements $x$ and $\eta$ are comparable, and hence their meet and join are determined by their pseudo-order. The pairs inside $\{a,b,c\}$ have the same meet and join pattern as the cycle in $\mathbb{T}_5$. Therefore, $\mathbb{T}_{\mathrm{inf}}=(X, \trianglelefteq, \wedge, \vee, 0,1)$ is an infinite bounded proper trellis.

Define $O:X^2\to X$ by
$O(x,y)=r(x)r(y)$, where $r:X\to[0,1]$ is given by
$r(x)=x$ if $x\in[0,1]$,
$r(a)=r(b)=r(c)=1/2$.
Then one can easily show that $O$ is a quasi-overlap function on $\mathbb{T}_{\mathrm{inf}}$.

We now verify that $O$ has the limiting property. Let $x\in X\setminus\{0,1\}$ and put
$
q=r(x)\in(0,1).
$
We first prove that
\begin{equation}\label{eq8}x_O^{(n)}=q^n\mbox{ when }n\ge2.\end{equation} Indeed,
$
x_O^{(2)}=O(x,x)=r(x)^2=q^2.
$
Suppose that $x_O^{(n)}=q^n$ for some $n\ge2$. Since $q^n\in(0,1)$, we have $r(q^n)=q^n$, and hence
$
x_O^{(n+1)}
=O(x_O^{(n)},x)
=r(q^n)r(x)
=q^{n+1}.
$
Thus Eq.\eqref{eq8} follows by induction.

Now, let $y\in X\setminus\{0\}$. The rest proof are divided into two cases as follows.

Case 1. If $y\in(0,1]$, then choose $N\ge2$ such that
$
q^N<y.
$
Thus by Eq.\eqref{eq8}, for every $n\ge N$,
$
x_O^{(n)}=q^n\le q^N<y.
$
Since $x_O^{(n)},y\in(0,1]$, it follows that
$
x_O^{(n)}\triangleleft y.
$

 Case 2. If $y\in\{a,b,c\}$, then choose $N\ge2$ such that
$
q^N<1/2.
$
Then by Eq.\eqref{eq8}, for every $n\ge N$,
$
x_O^{(n)}=q^n<1/2.
$
Thus by the definition of the pseudo-order $\trianglelefteq$,
$
x_O^{(n)}\triangleleft y.
$

Cases 1 and 2 imply that $O$ satisfies the limiting property. In particular, $O$ is Archimedean by Proposition~\ref{prop:limiting-implies-arch}.
\end{example}

\section{Construction methods of quasi-overlap functions}\label{sec:constructions}

In this section, we establish two constructing methods of quasi-overlap functions, which are based on boundary-faithful mappings and adjunctions on bounded psosets and the transitive-endpoints of a bounded trellis, respectively.
\subsection{Based on boundary-faithful mappings and adjunctions}

We generalize the $0,1$-homomorphism in \cite{Qiao2022} as follows.
\begin{definition}\label{def:boundary-map}
\emph{Let $\mathbb{P}=(X,\trianglelefteq,0_X,1_X)$ and $\mathbb{Q}=(Y,\trianglelefteq,0_Y,1_Y)$ be bounded psosets. A mapping $f:X\to Y$ is called a \emph{boundary-faithful increasing mapping} if, for all $x,y\in X$,
\begin{enumerate}[label=(\roman*)]
\item $x\trianglelefteq y$ implies $f(x)\trianglelefteq f(y)$;
\item $f(x)=0_Y$ if and only if $x=0_X$;
\item $f(x)=1_Y$ if and only if $x=1_X$.
\end{enumerate}}
\end{definition}

Notice that because the proof of Lemma 3.1 in \cite{Qiao2022} does not need the transitivity, we have the following theorem.
\begin{theorem}\label{thm:boundary-construct}
Let $\mathbb{P}=(X,\trianglelefteq,0_X,1_X)$ and $\mathbb{Q}=(Y,\trianglelefteq,0_Y,1_Y)$ be bounded psosets. Let $d:X\to Y$ and $g:Y\to X$ be boundary-faithful increasing mappings, and $O_Y$ be a quasi-overlap function on $\mathbb{Q}$. Then the binary operation $O_{d,g}:X^2\to X$ defined by
$$
 O_{d,g}(x,y)=g\bigl(O_Y(d(x),d(y))\bigr)
$$
is a quasi-overlap function on $\mathbb{P}$.
\end{theorem}

We now connect Theorem~\ref{thm:boundary-construct} with the adjunction on psosets introduced in \cite{Geng2026}.
\begin{definition}[\cite{Geng2026}]\label{def:psoset-Adjunction}
\emph{Let $(X,\trianglelefteq)$ and $(Y,\trianglelefteq)$ be psosets, and let $d:X\to Y$ and $g:Y\to X$ be mappings. The pair $(g,d)$ is called an \emph{adjunction} between $X$ and $Y$ if both $d$ and $g$ are increasing and, for any $x\in X$, $y\in Y$,
$$
 y\trianglelefteq d(x)
 \Longleftrightarrow
 g(y)\trianglelefteq x.
$$
When the psosets $X$ and $Y$ are same, the pair $(g,d)$ is called an \emph{adjunction} on $X$. }
\end{definition}

The following proposition provides sufficient conditions ensuring that both mappings in an adjunction are boundary-faithful increasing.
\begin{proposition}\label{prop:Adjunction-boundary}
Let $\mathbb{P}=(X,\trianglelefteq,0_X,1_X)$ and $\mathbb{Q}=(Y,\trianglelefteq,0_Y,1_Y)$ be bounded psosets, and $(g,d)$ be an adjunction between $X$ and $Y$, where $d:X\to Y$ and $g:Y\to X$. Assume that
\begin{enumerate}[label=(\roman*)]
\item $d$ is surjective;
\item $d(x)\neq0_Y$ for any $x\neq0_X$;
\item $d(x)\neq1_Y$ for any $x\neq1_X$.
\end{enumerate}
Then
$d\circ g=\operatorname{id}_Y$,
$g$ is injective, and both $d$ and $g$ are boundary-faithful increasing mappings.
\end{proposition}

\begin{proof}
We first prove that
$d\circ g=\operatorname{id}_Y.$ For any $y\in Y$, we have $g(y)\trianglelefteq g(y)$, and then
\begin{equation}\label{eq2}
 y\trianglelefteq d(g(y)).
\end{equation}
Since $d$ is surjective, there exists an $x\in X$ such that $d(x)=y$. From
$
 d(x)\trianglelefteq d(x)
$
we obtain
$$
 g(d(x))\trianglelefteq x.
$$
Applying the increasingness of $d$ gives
$$
 d(g(d(x)))\trianglelefteq d(x),
$$
i.e.,
$$
 d(g(y))\trianglelefteq y,
$$
which together with Eq.(\ref{eq2}) means that
$$
 d(g(y))=y.
$$
Thus $d\circ g=\operatorname{id}_Y$. Consequently, $g$ is injective.

We next verify the boundary conditions for $d$. Since $d$ is surjective,
there are $x_0,x_1\in X$ such that
$
 d(x_0)=0_Y
 \mbox{ and }
 d(x_1)=1_Y.
$
Assumptions (ii) and (iii) imply, respectively,
$
 x_0=0_X
 \mbox{ and }
 x_1=1_X.
$
Together with assumptions~(ii) and~(iii), we obtain
$$
 d(x)=0_Y\Longleftrightarrow x=0_X
$$
and
$$
 d(x)=1_Y\Longleftrightarrow x=1_X.
$$
Therefore, $d$ is a boundary-faithful increasing mapping.

It remains to verify the boundary conditions for $g$. Since
$d\circ g=\operatorname{id}_Y$,
$
 d(g(0_Y))=0_Y\mbox{ implies } g(0_Y)=0_X.
$
Similarly,
$
 d(g(1_Y))=1_Y\mbox{ implies } g(1_Y)=1_X.
$

Conversely, if $g(y)=0_X$, then
$
 y=d(g(y))=d(0_X)=0_Y.
$
Hence
$$
 g(y)=0_X\Longleftrightarrow y=0_Y.
$$
Likewise, we have
$$
 g(y)=1_X\Longleftrightarrow y=1_Y.
$$

Therefore, $g$ is also a boundary-faithful increasing mapping.
\end{proof}

Therefore, Theorem \ref{thm:boundary-construct} implies the following one.
\begin{theorem}\label{thm:Adjunction-qo}
Under the assumptions of Proposition~\ref{prop:Adjunction-boundary}, let $O_Y$ be a quasi-overlap function on $\mathbb{Q}$. Define $O_{d,g}:X^2\to X$ by
$$
 O_{d,g}(x,y)
 =
 g\bigl(O_Y(d(x),d(y))\bigr).
$$
Then $O_{d,g}$ is a quasi-overlap function on $\mathbb{P}$.
Moreover, the following diagrams commute:
\begin{center}
\begin{tikzpicture}[>=stealth,baseline=(current bounding box.center)]
\node (YY) at (0,1.6) {$Y\times Y$};
\node (Y)  at (3.2,1.6) {$Y$};
\node (XX) at (0,0) {$X\times X$};
\node (X)  at (3.2,0) {$X$};

\draw[->] (YY) -- node[above] {$O_Y$} (Y);
\draw[->] (YY) -- node[left] {$g\times g$} (XX);
\draw[->] (Y)  -- node[right] {$g$} (X);
\draw[->] (XX) -- node[below] {$O_{d,g}$} (X);
\end{tikzpicture}
\qquad\qquad
\begin{tikzpicture}[>=stealth,baseline=(current bounding box.center)]
\node (XX) at (0,1.6) {$X\times X$};
\node (X)  at (3.2,1.6) {$X$};
\node (YY) at (0,0) {$Y\times Y$};
\node (Y)  at (3.2,0) {$Y$};

\draw[->] (XX) -- node[above] {$O_{d,g}$} (X);
\draw[->] (XX) -- node[left] {$d\times d$} (YY);
\draw[->] (X)  -- node[right] {$d$} (Y);
\draw[->] (YY) -- node[below] {$O_Y$} (Y);
\end{tikzpicture}
\end{center}
Thus $O_{d,g}$ is an extension of $O_Y$ through the injective
mapping $g$.
\end{theorem}

\begin{remark}\label{rem:galois-s}
\emph{Theorem~\ref{thm:Adjunction-qo} extends Theorem 3.1 of \cite{Qiao2022}.}
\end{remark}

In particular, we have the following corollary which tells us that a quasi-overlap function on bounded psosets can be induced by a given quasi-overlap function through an adjunction.
\begin{corollary}\label{cor:poset-to-psoset-qo}
Let $\mathbb{P}=(X,\trianglelefteq,0_X,1_X)$ be a bounded psoset and $\mathbb{Q}=(Y,\leq,0_Y,1_Y)$ be a bounded poset. Let $(g,d)$ be an adjunction between $X$ and $Y$, where $d:X\to Y$ and $g:Y\to X$ satisfying the assumptions (i)--(iii) of Proposition~\ref{prop:Adjunction-boundary}. If $O_Y$ is a quasi-overlap function on $\mathbb{Q}$ in the sense of \cite{Qiao2022}, then
$$
 O_X(x,y)
 =
 g\bigl(O_Y(d(x),d(y))\bigr)
$$
is a quasi-overlap function on $\mathbb{P}$ in the sense of Definition~\ref{def:qo}.
\end{corollary}

\begin{proof}
From Theorem~\ref{thm:Adjunction-qo}, it is enough to verify the monotonicity of $O_Y$. Let
$$
 u\leq u'
 \mbox{ and }
 v\leq v'.
$$
By the commutativity and one-coordinate monotonicity,
$$
 O_Y(u,v)
 =
 O_Y(v,u)
 \leq
 O_Y(v,u')
 =
 O_Y(u',v),
$$
and
$$
 O_Y(u',v)\leq O_Y(u',v').
$$
Thus
$$
 O_Y(u,v)\leq O_Y(u',v').
$$
\end{proof}

The preceding methods construct quasi-overlap functions between two bounded psosets. We now return to construct a quasi-overlap function on a fixed bounded trellis, originating from \cite{Monteiro2024}.
\begin{definition}\label{def:boundary-retraction}
\emph{Let $\mathbb{P}=(X,\trianglelefteq,0,1)$ be a bounded psoset and $R\subseteq X$ with $0,1\in R$. A mapping $r:X\to R$ is called a \emph{boundary-faithful increasing retraction} if $r:X\to R$ is a boundary-faithful increasing mapping and $r(x)=x$ for any $x\in R.$}
\end{definition}

The following construction is an intrinsic specialization of Theorem~\ref{thm:boundary-construct}.
\begin{theorem}\label{thm:retraction}
Let $\mathbb{T}=(X,\trianglelefteq,\wedge,\vee,0,1)$ be a bounded trellis and $R$ be a subtrellis of $\mathbb{T}$ containing $0$ and $1$. Let $r:X\to R$ be a boundary-faithful increasing retraction, and $O_R$ be a quasi-overlap function on $R$. Then $O:X^2\to X$ defined by
$$
 O(x,y)=O_R(r(x),r(y))
$$
is a quasi-overlap function on $\mathbb{T}$. Moreover,
$
 O|_{R\times R}=O_R.
$
\end{theorem}

\begin{proof}
Let the mapping $i:R\to X$ satisfy that $i(x)=x$ for any $x\in R$ (we also call $i$ an \emph{inclusion mapping}). Since $R$ contains $0$ and $1$, the mapping $i$ is boundary-faithful increasing. Applying Theorem~\ref{thm:boundary-construct} shows that the operation $\widetilde O:X^2\to X$ defined by
$$
 \widetilde O(x,y)
 =
 i\bigl(O_R(r(x),r(y))\bigr)
$$
is a quasi-overlap function on $\mathbb{T}$. Identifying $R$ with its image under $i$, we have $\widetilde O=O$. Finally, if $x,y\in R$, then $r(x)=x$ and $r(y)=y$ since $r:X\to R$ is a boundary-faithful increasing retraction, so
$$
 O(x,y)=O_R(x,y).
$$
Hence $O|_{R\times R}=O_R$.
\end{proof}

Recall from \cite{Zedam2023} that an \emph{interior operator} on a $\wedge$-semi-trellis $\mathbb{T}=(X,\trianglelefteq,\wedge)$ is a mapping $I:X\to X$ satisfying, for all $x,y\in X$,
\begin{enumerate}[label=(\roman*)]
\item $I(x)\trianglelefteq x$;
\item $I(I(x))=I(x)$;
\item $I(x\wedge y)=I(x)\wedge I(y)$.
\end{enumerate}
 Obviously, such a mapping is increasing, and its range is a $\wedge$-subtrellis.

\begin{corollary}\label{cor:interior}
Let $\mathbb{T}=(X,\trianglelefteq,\wedge,\vee,0,1)$ be a bounded trellis, and $I:X\to X$ be an interior operator on $\mathbb{T}$.
Assume that
\begin{enumerate}[label=(\roman*)]
\item $I(1)=1$;
\item $I(x)=0$ if and only if $x=0$;
\item $R_I=I(X)$ is a bounded subtrellis of $\mathbb{T}$.
\end{enumerate}
If $O_{R_I}$ is a quasi-overlap function on $R_I$, then the binary
operation $O:X^2\to X$ defined by
$$
 O(x,y)=O_{R_I}(I(x),I(y))
$$
is a quasi-overlap function on $\mathbb{T}$.
\end{corollary}

\begin{proof}
Since $R_I=I(X)$, for any $x\in R_I$ there exists a $y\in X$ such that $I(y)=x$. Then for any $x\in R_I$, $I(I(y))=I(x)$. Thus $I(x)=I(y)=x$ since $I$ is idempotent, i.e., $I(x)=x$ for any $x\in R_I$.

By assumption,
$
 I(x)=0\Longleftrightarrow x=0.
$
If $I(x)=1$, then $
 1=I(x)\trianglelefteq x
$
since $I:X\to X$ be an interior operator on $\mathbb{T}$, which implies $x=1$. Thus
$
 I(x)=1\Longleftrightarrow x=1.
$

Therefore, $I$ is boundary faithful increasing retraction, and the conclusion follows from Theorem~\ref{thm:retraction}.
\end{proof}

\begin{remark}\label{rem:Adjunction-retraction}
\emph{Let $R$ be a bounded subtrellis of $\mathbb{T}=(X,\trianglelefteq,\wedge,\vee,0,1)$, $r:X\to R$ be a boundary-faithful increasing retraction and $i:R\to X$ be the inclusion mapping. Take $d=r$ and $g=i$.
Then
$r\circ i=\operatorname{id}_R.$
If $(i,r)$ is further an adjunction between $X$ and $R$, then
Theorem~\ref{thm:Adjunction-qo} gives
$
 O_{r,i}(x,y)
 =
 i\bigl(O_R(r(x),r(y))\bigr)
 =
 O_R(r(x),r(y))
$
for any $x,y\in X$, which is exactly the construction in
Theorem~\ref{thm:retraction}. However, a boundary-faithful increasing retraction need not form an adjunction with the inclusion mapping. For example, consider the bounded proper trellis $\mathbb{T}_5$ in Example \ref{examp:T5-classification} and let
$
R=\{0,a,1\}
$
and $i:R\to X$ be the inclusion mapping. Define $r: X\to R$ by
$
r(0)=0, r(a)=r(b)=r(c)=a, r(1)=1.
$
Then $R$ is a bounded subtrellis of $\mathbb{T}_5$, and $r$ is a
boundary-faithful increasing retraction.
Taking $x=c$ and $y=a$, we have
$
a\trianglelefteq r(c)=a,
\mbox{ but }
i(a)=a\ntrianglelefteq c.
$
Hence
$
y\trianglelefteq r(x)
\Longleftrightarrow
i(y)\trianglelefteq x
$
does not hold, and therefore $(i,r)$ is not an adjunction
between $X$ and $R$.}
\end{remark}

\subsection{A transitive-endpoint gluing construction}
We now give the transitive-endpoint gluing construction of quasi-overlap functions on a
bounded trellis. We first need the following lemma.
\begin{lemma}\label{lem:interval-subtrellis}
Let $\mathbb{T}=(X,\trianglelefteq,\wedge,\vee,0,1)$ be a
bounded trellis. Let $n\ge2$ and
$e_1,\ldots,e_{n-1}\in X^{\rm tr}\setminus\{0,1\}$ satisfy
$
0=e_0\triangleleft e_1\triangleleft\cdots
\triangleleft e_{n-1}\triangleleft e_n=1.
$
Assume that
$
X=\bigcup_{i=1}^{n}I_i,
I_i=[e_{i-1},e_i]
=\{x\in X:e_{i-1}\trianglelefteq x\trianglelefteq e_i\}.
$
For $i=1,\ldots,n$, put
$
A_i=I_i\setminus\{e_{i-1}\}.
$
Then the following statements hold:
\begin{enumerate}[label=(\roman*)]
\item
Each $I_i$ is a bounded subtrellis of $\mathbb{T}$ with
the smallest element $e_{i-1}$ and the greatest element $e_i$.

\item
The sets $A_1,\ldots,A_n$ form a partition of
$X\setminus\{0\}$.

\item
If $x\in A_i$ and $y\in A_j$ with $i<j$, then
$x\triangleleft y$.
\end{enumerate}
\end{lemma}

\begin{proof}
By successive applications of the middle-transitivity
of the internal endpoints,
$
e_i\triangleleft e_j
$
 for all $0\le i<j\le n$. These points are pairwise distinct. Indeed, if $e_i=e_j$
for some $i<j$, then
$e_i\trianglelefteq e_{i+1}\trianglelefteq e_j=e_i$,
so the antisymmetry would give $e_i=e_{i+1}$, a contradiction.

\emph{(i)}
Fix $i\in\{1,\ldots,n\}$ and let $x,y\in I_i$.
Since $e_{i-1}$ is a lower bound of both $x$ and $y$, while
$e_i$ is their upper bound, the definitions of meet and join give
$
e_{i-1}\trianglelefteq x\wedge y,
x\vee y\trianglelefteq e_i.
$
Moreover,
$
x\wedge y\trianglelefteq x\trianglelefteq e_i.
$
If $i<n$, the left-transitivity of $e_i$ gives
$x\wedge y\trianglelefteq e_i$. If $i=n$, then $x\wedge y\trianglelefteq e_i$ since $e_n=1$. Thus \begin{equation}\label{eq3}e_{i-1}\trianglelefteq x\wedge y\trianglelefteq e_i.\end{equation}

On the other hand,
$
e_{i-1}\trianglelefteq x\trianglelefteq x\vee y.
$
If $i>1$, then the right-transitivity of $e_{i-1}$ gives
$e_{i-1}\trianglelefteq x\vee y$. If $i=1$, then $e_{i-1}\trianglelefteq x\vee y$ since $e_0=0$.
Hence $x\wedge y,x\vee y\in I_i$. Thus \begin{equation}\label{eq4}e_{i-1}\trianglelefteq x\vee y\trianglelefteq e_i.\end{equation}

Eqs. (\ref{eq3}) and (\ref{eq4}) imply that $I_i$ is a bounded
subtrellis of $X$ with the smallest element $e_{i-1}$ and the greatest
element $e_i$.

\emph{(ii)}
Let $i<j$. If $j=i+1$, then
$
I_i\cap I_{i+1}=\{e_i\}.
$
If $j\ge i+2$ and $z\in I_i\cap I_j$, then
$
e_{j-1}\trianglelefteq z\trianglelefteq e_i.
$
The right-transitivity of $e_{j-1}$ gives
$e_{j-1}\trianglelefteq e_i$.
On the other hand, $e_i\trianglelefteq e_{j-1}$.
The antisymmetry gives $e_i=e_{j-1}$,
contradicting the distinctness of the endpoints.
Thus
$
I_i\cap I_j=\emptyset
$
when $j\ge i+2$. Therefore, the sets $A_1,\ldots,A_n$ are non-empty and pairwise disjoint, which together with $X=\bigcup_{i=1}^{n}I_i$ means that
$A_1,\ldots,A_n$ form a partition of $X\setminus\{0\}$.

\emph{(iii)}
Let $x\in A_i$ and $y\in A_j$ with $i<j$.
We have $x\trianglelefteq e_i$.
If $j=i+1$, then $e_i=e_{j-1}\trianglelefteq y$, i.e., $e_i\trianglelefteq y$.
If $j\ge i+2$, then
$
e_i\triangleleft e_{j-1}\trianglelefteq y,
$
and the right-transitivity of $e_i$ gives
$e_i\trianglelefteq y$.
In either case,
$
x\trianglelefteq e_i\trianglelefteq y.
$
The middle-transitivity of $e_i$ yields
$x\trianglelefteq y$. Thus $x\triangleleft y$ since $x=y$ contradicts the fact
$A_i\cap A_j=\emptyset$.
\end{proof}

The following theorem presents an important constructing method of a quasi-overlap function on a
bounded trellis.
\begin{theorem}\label{thm:gluing}
Under the assumptions of Lemma~\ref{lem:interval-subtrellis},
let $O_i:I_i^2\to I_i$ be a quasi-overlap function on
$I_i$ for every $i=1,\ldots,n$.
Define $O:X^2\to X$ by
$$
O(x,y)=
\begin{cases}
0,
& x=0\mbox{ or }y=0,\\[1mm]
O_i(x,y),
& x,y\in A_i\mbox{ for some }i,\\[1mm]
e_{\min\{i,j\}},
& x\in A_i\mbox{ and }\ y\in A_j\mbox{ with }\ i\neq j.
\end{cases}
$$
Then $O$ is a quasi-overlap function on $\mathbb{T}$.
Moreover,
$
O|_{I_i^2}=O_i
$
for every $i=1,\ldots,n$.
\end{theorem}

\begin{proof}
First note that by Lemma~\ref{lem:interval-subtrellis} (ii), the sets
$A_1,\ldots,A_n$ form a partition of $X\setminus\{0\}$.
Hence the operation $O$ is well defined.

Now, let $x,y\neq0$ and
$x\in A_i$, $y\in A_j$. Put
$
m=\min\{i,j\}.
$
If $i=j=m$, then
$
O(x,y)=O_m(x,y).
$
From the construction of every $A_i$, $x,y\neq e_{m-1}$, so that (QO2) for $O_m$ gives
$
O_m(x,y)\neq e_{m-1}.
$
As $O_m(x,y)\in I_m$, it follows that
$
O(x,y)\in A_m.
$
If $i\neq j$, then from the construction of $O$,
$
O(x,y)=e_m\in A_m.
$
Therefore,
\begin{equation}\label{eq5}
 x\in A_i,\ y\in A_j
 \Longrightarrow
 O(x,y)\in A_{\min\{i,j\}}.
\end{equation}

In what follows, we prove that $O$ is a quasi-overlap function on $\mathbb{T}$ according to Definition \ref{def:qo}.

\emph{(QO1)}
The commutativity of $O$ follows immediately from the
commutativity of the functions $O_i$ and the symmetry of
$\min\{i,j\}$.

\emph{(QO2)}
We split the proof of (QO2) into two parts.

Part 1. If $x=0$ or $y=0$, then from the construction of $O$ the first case gives
$$
 O(x,y)=0.
$$

Part 2. Suppose that $O(x,y)=0$. Assume, to the contrary, that
$x,y\neq0$. By Lemma~\ref{lem:interval-subtrellis} (ii), there are
unique indices $i,j\in\{1,\ldots,n\}$ such that
$
x\in A_i
$
and
$
y\in A_j.
$
By Eq.\eqref{eq5},
$$
 O(x,y)\in A_{\min\{i,j\}}.
$$
Thus $O(x,y)\neq 0$ since $A_{\min\{i,j\}}\subseteq X\setminus\{0\}$, a contradiction.

Therefore, by Parts 1 and 2,
$$
 O(x,y)=0
 \Longleftrightarrow
 x=0\mbox{ or }y=0.
$$

\emph{(QO3)}
The proof will be completed by two parts.

Part A. Suppose that $O(x,y)=1$. By (QO2), we have
$x,y\neq0$. Let
$
x\in A_i
$
and
$
y\in A_j.
$
Then by Eq.\eqref{eq5},
$$
 1=O(x,y)\in A_{\min\{i,j\}}.
$$
Since $1\in A_n$ and the sets $A_1,\ldots,A_n$ are pairwise
disjoint, we obtain
$$
 \min\{i,j\}=n.
$$
Thus $i=j=n$, and hence
$$
 O(x,y)=O_n(x,y).
$$
By (QO3) for $O_n$,
$$
 O_n(x,y)=1
 \Longrightarrow
 x=y=1.
$$

Part B. If $x=y=1$, then from $1\in A_n$, we have that $x,y\in A_n$. Thus from the construction of $O$,
$$
 O(1,1)=O_n(1,1)=1.
$$

Consequently,
$$
 O(x,y)=1
 \Longleftrightarrow
 x=y=1.
$$

\emph{(QO4)}
Let
$
(x,y)\trianglelefteq_2(x',y').
$
If $x=0$ or $y=0$, then by (QO2),
$$
 O(x,y)=0\trianglelefteq O(x',y').
$$

Now, assume that $x,y\neq0$. Since
$
x\trianglelefteq x'
$
and
$
y\trianglelefteq y',
$
we also have $x',y'\neq0$. Then by Lemma~\ref{lem:interval-subtrellis} (ii), there are unique
indices $i,j,i',j'$ with $i\leq i'$ and $j\leq j'$ such that
$x\in A_i, y\in A_j, x'\in A_{i'}$ and $y'\in A_{j'}$. Put
$
 m=\min\{i,j\},
 m'=\min\{i',j'\}.
$
Then $m\leq m'$. We distinguish two cases as follows.

Case 1. If $m<m'$, then by Eq.(\ref{eq5}),
$
 O(x,y)\in A_m
 \mbox{ and }
 O(x',y')\in A_{m'}.
$
Hence Lemma~\ref{lem:interval-subtrellis} (iii) gives
$$
 O(x,y)\triangleleft O(x',y').
$$

Case 2. If $m=m'$, then suppose that $i=j=m$.
There are two subcases as follows.

Subcase 1. If $i'=j'=m$, then $x,y,x',y'\in I_m$, and the
increasingness of $O_m$ gives
$$
 O(x,y)
 =
 O_m(x,y)
 \trianglelefteq
 O_m(x',y')
 =
 O(x',y').
$$

Subcase 2. If one of $i'$ and $j'$ is greater than $m$, then
$\min\{i',j'\}=m$ implies $i'\neq j'$. Hence from the construction of $O$,
$$
 O(x',y')=e_m.
$$
On the other hand, from $i=j=m$ and the construction of $O$ we have that
$$
 O(x,y)
 =
 O_m(x,y)
 \trianglelefteq
 e_m
 =
 O(x',y').
$$
since $O_m(x,y)\in I_m$ and $e_m$ is the greatest element of
$I_m$.

Now, suppose that $i\neq j$, say $i<j$. Then $i=m<j\leq j'$.
Thus $i'=m$ since $\min\{i',j'\}=m$.
Consequently,  from the construction of $O$,
$$
 O(x,y)=e_m=O(x',y').
$$

Therefore, in all cases,
$$
 O(x,y)\trianglelefteq O(x',y').
$$

Finally, we prove $O|_{I_i^2}=O_i$ for every $i=1,\ldots,n$. Fix
$i\in\{1,\ldots,n\}$ and let $x,y\in I_i$. We consider two cases as follows.

Case 1. If $x,y\in A_i$, then, by the construction of $O$,
$$
O(x,y)=O_i(x,y).
$$

Case 2. If $x\notin A_i$ or $y\notin A_i$, then $x=e_{i-1}$ or $y=e_{i-1}$. Then by (QO2),
$$
O_i(x,y)=e_{i-1}
$$
since $O_i$ is a quasi-overlap function. There are two subcases as follows.

If $i=1$, then $e_{i-1}=0$, and hence
$O(x,y)=0=O_i(x,y)$ by the construction of $O$.

If $i>1$, then $e_{i-1}\in A_{i-1}$. Thus, if exactly one of
$x,y$ equals $e_{i-1}$, then by the construction of $O$, the cross-interval case gives
$O(x,y)=e_{i-1}$; if $x=y=e_{i-1}$, then by the construction of $O$,
$$
O(x,y)=O_{i-1}(x,y).
$$

Cases 1 and 2 imply that
$$
O(x,y)=O_i(x,y).
$$
\end{proof}

\begin{remark}\label{rem:gluing-comparison}
\emph{In general, the function $O$ constructed in Theorem 3.3 of both \cite{WangHu2022} and \cite{GengZhangQiao2025} may not be quasi-overlap on a bounded trellis since the increasingness of $O$ will be false when $(x,y)\notin [e_{i-1}, e_i]^2$ shown as Example \ref{exam2.7}. Therefore, Theorem \ref{thm:gluing} is different from Theorem 3.3 of both \cite{WangHu2022} and \cite{GengZhangQiao2025} in essential.}
\end{remark}

\begin{example}\label{ex}
\upshape
Let $\mathbb{T}_{16}=(X,\trianglelefteq,\wedge,\vee,0,1)$ be a bounded trellis (see Figure~\ref{fig:three-cycle-trellis}), where
$X=\{0,a,b,c,d,e,f,g,h,i,j,k,l,m,n,1\}$. Then
$f,j\in X^{\rm tr}\setminus\{0,1\}$ and
$X=[0,f]\cup[f,j]\cup[j,1]$. Let $O_{a,b}$, $O_{g,h}$ and $O_{k,l}$ be the quasi-overlap functions on $[0,f]$, $[f,j]$ and $[j,1]$, respectively (see Tables 5, 6 and 7). Then from Theorem~\ref{thm:gluing}, the operation $O$ shown by Table \ref{tab:Oeta4} is a quasi-overlap function.

\begin{figure}[!h]
\centering
\begin{tikzpicture}[scale=0.9]

\node[circle,draw,inner sep=2pt] (0) at (0,0) {$0$};

\node[circle,draw,inner sep=2pt] (a) at (0,1.2) {$a$};
\node[circle,draw,inner sep=2pt] (b) at (-1.15,2.0) {$b$};
\node[circle,draw,inner sep=2pt] (c) at (-0.7,3.3) {$c$};
\node[circle,draw,inner sep=2pt] (d) at (0.7,3.3) {$d$};
\node[circle,draw,inner sep=2pt] (e) at (1.15,2.0) {$e$};

\node[circle,draw,inner sep=2pt] (f) at (0,4.7) {$f$};

\node[circle,draw,inner sep=2pt] (g) at (-1,6.2) {$g$};
\node[circle,draw,inner sep=2pt] (h) at (0,6.6) {$h$};
\node[circle,draw,inner sep=2pt] (i) at (1,6.2) {$i$};

\node[circle,draw,inner sep=2pt] (j) at (0,8.0) {$j$};

\node[circle,draw,inner sep=2pt] (k) at (-1.1,9.4) {$k$};
\node[circle,draw,inner sep=2pt] (l) at (0,10.1) {$l$};
\node[circle,draw,inner sep=2pt] (m) at (1.1,9.4) {$m$};
\node[circle,draw,inner sep=2pt] (n) at (0,8.8) {$n$};

\node[circle,draw,inner sep=2pt] (1) at (0,11.6) {$1$};

\draw (0)--(a);
\draw (0)--(b);
\draw (0)--(c);
\draw (0)--(d);
\draw (0)--(e);

\draw (a)--(f);
\draw (b)--(f);
\draw (c)--(f);
\draw (d)--(f);
\draw (e)--(f);

\draw[->,bend left=20] (a) to (b);
\draw[->,bend left=20] (b) to (c);
\draw[->,bend left=20] (c) to (d);
\draw[->,bend left=20] (d) to (e);
\draw[->,bend left=20] (e) to (a);

\draw (f)--(g);
\draw (f)--(h);
\draw (f)--(i);

\draw (g)--(j);
\draw (h)--(j);
\draw (i)--(j);

\draw[->,bend left=20] (g) to (h);
\draw[->,bend left=20] (h) to (i);
\draw[->,bend left=20] (i) to (g);

\draw (j)--(k);
\draw (j)--(l);
\draw (j)--(m);
\draw (j)--(n);

\draw (k)--(1);
\draw (l)--(1);
\draw (m)--(1);
\draw (n)--(1);

\draw[->,bend left=20] (k) to (l);
\draw[->,bend left=20] (l) to (m);
\draw[->,bend left=20] (m) to (n);
\draw[->,bend left=20] (n) to (k);

\end{tikzpicture}
\caption{The bounded trellis $\mathbb{T}_{16}$.}
\label{fig:three-cycle-trellis}
\end{figure}
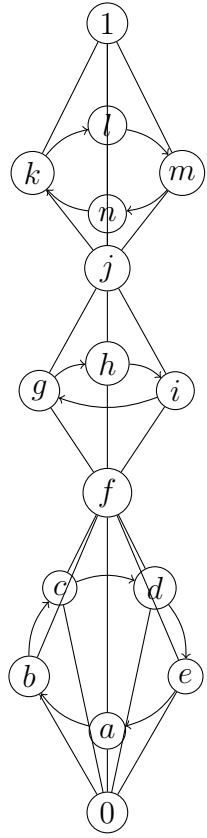
\begin{table}[!h]
\centering
\caption{The quasi-overlap function $O_{a,b}$ on $[0,f]$.}
\label{tab:$O_{a,b}$ on $[0,f]$}
\begin{tabular}{c|ccccccc}
$O_{a,b}$&$0$&$a$&$b$&$c$&$d$&$e$&$f$\\\hline
$0$&$0$&$0$&$0$&$0$&$0$&$0$&$0$\\
$a$&$0$&$a$&$a$&$a$&$a$&$a$&$b$\\
$b$&$0$&$a$&$a$&$a$&$a$&$a$&$b$\\
$c$&$0$&$a$&$a$&$a$&$a$&$a$&$b$\\
$d$&$0$&$a$&$a$&$a$&$a$&$a$&$b$\\
$e$&$0$&$a$&$a$&$a$&$a$&$a$&$b$\\
$f$&$0$&$b$&$b$&$b$&$b$&$b$&$f$\\
\end{tabular}
\end{table}

\begin{table}[!h]
\centering
\caption{The quasi-overlap function $O_{g,h}$ on $[f,j]$.}
\label{tab:$O_{g,h}$ on $[f,j]$}
\begin{tabular}{c|ccccc}
$O_{g,h}$&$f$&$g$&$h$&$i$&$j$\\\hline
$f$&$f$&$f$&$f$&$f$&$f$\\
$g$&$f$&$g$&$g$&$g$&$h$\\
$h$&$f$&$g$&$g$&$g$&$h$\\
$i$&$f$&$g$&$g$&$g$&$h$\\
$j$&$f$&$h$&$h$&$h$&$j$\\
\end{tabular}
\end{table}

\begin{table}[!h]
\centering
\caption{The quasi-overlap function $O_{k,l}$ on $[j,1]$.}
\label{tab:$O_{k,l}$ on $[j,1]$}
\begin{tabular}{c|cccccc}
$O_{k,l}$&$j$&$k$&$l$&$m$&$n$&$1$\\\hline
$j$&$j$&$j$&$j$&$j$&$j$&$j$\\
$k$&$j$&$k$&$k$&$k$&$k$&$l$\\
$l$&$j$&$k$&$k$&$k$&$k$&$l$\\
$m$&$j$&$k$&$k$&$k$&$k$&$l$\\
$n$&$j$&$k$&$k$&$k$&$k$&$l$\\
$1$&$j$&$l$&$l$&$l$&$l$&$1$\\
\end{tabular}
\end{table}
\begin{table}[!h]
\centering
\caption{The operation $O$ on $\mathbb{T}_{16}$.}
\label{tab:Oeta4}
\begin{tabular}{c|cccccccccccccccc}
$O$ & $0$ & $a$ & $b$ & $c$ & $d$ & $e$ & $f$ & $g$ & $h$ & $i$ & $j$ & $k$ & $l$ & $m$ & $n$ & $1$\\
\hline
$0$ & $0$ & $0$ & $0$ & $0$ & $0$ & $0$ & $0$ & $0$ & $0$ & $0$ & $0$ & $0$ & $0$ & $0$ & $0$ & $0$\\
$a$ & $0$ & $a$ & $a$ & $a$ & $a$ & $a$ & $b$ & $f$ & $f$ & $f$ & $f$ & $f$ & $f$ & $f$ & $f$ & $f$\\
$b$ & $0$ & $a$ & $a$ & $a$ & $a$ & $a$ & $b$ & $f$ & $f$ & $f$ & $f$ & $f$ & $f$ & $f$ & $f$ & $f$\\
$c$ & $0$ & $a$ & $a$ & $a$ & $a$ & $a$ & $b$ & $f$ & $f$ & $f$ & $f$ & $f$ & $f$ & $f$ & $f$ & $f$\\
$d$ & $0$ & $a$ & $a$ & $a$ & $a$ & $a$ & $b$ & $f$ & $f$ & $f$ & $f$ & $f$ & $f$ & $f$ & $f$ & $f$\\
$e$ & $0$ & $a$ & $a$ & $a$ & $a$ & $a$ & $b$ & $f$ & $f$ & $f$ & $f$ & $f$ & $f$ & $f$ & $f$ & $f$\\
$f$ & $0$ & $b$ & $b$ & $b$ & $b$ & $b$ & $f$ & $f$ & $f$ & $f$ & $f$ & $f$ & $f$ & $f$ & $f$ & $f$\\
$g$ & $0$ & $f$ & $f$ & $f$ & $f$ & $f$ & $f$ & $g$ & $g$ & $g$ & $h$ & $j$ & $j$ & $j$ & $j$ & $j$\\
$h$ & $0$ & $f$ & $f$ & $f$ & $f$ & $f$ & $f$ & $g$ & $g$ & $g$ & $h$ & $j$ & $j$ & $j$ & $j$ & $j$\\
$i$ & $0$ & $f$ & $f$ & $f$ & $f$ & $f$ & $f$ & $g$ & $g$ & $g$ & $h$ & $j$ & $j$ & $j$ & $j$ & $j$\\
$j$ & $0$ & $f$ & $f$ & $f$ & $f$ & $f$ & $f$ & $h$ & $h$ & $h$ & $j$ & $j$ & $j$ & $j$ & $j$ & $j$\\
$k$ & $0$ & $f$ & $f$ & $f$ & $f$ & $f$ & $f$ & $j$ & $j$ & $j$ & $j$ & $k$ & $k$ & $k$ & $k$ & $l$\\
$l$ & $0$ & $f$ & $f$ & $f$ & $f$ & $f$ & $f$ & $j$ & $j$ & $j$ & $j$ & $k$ & $k$ & $k$ & $k$ & $l$\\
$m$ & $0$ & $f$ & $f$ & $f$ & $f$ & $f$ & $f$ & $j$ & $j$ & $j$ & $j$ & $k$ & $k$ & $k$ & $k$ & $l$\\
$n$ & $0$ & $f$ & $f$ & $f$ & $f$ & $f$ & $f$ & $j$ & $j$ & $j$ & $j$ & $k$ & $k$ & $k$ & $k$ & $l$\\
$1$ & $0$ & $f$ & $f$ & $f$ & $f$ & $f$ & $f$ & $j$ & $j$ & $j$ & $j$ & $l$ & $l$ & $l$ & $l$ & $1$\\
\end{tabular}
\end{table}
\end{example}
\section{Concluding remarks }\label{sec:conclusion}

This article extended the theory of quasi-overlap and quasi-grouping functions to bounded psosets, where the underlying relation is not necessarily transitive. The results obtained show that aggregation operators of overlap and grouping type can be studied in a substantially more general setting than bounded posets and lattices. They also reveal some essential differences between the pseudo-ordered setting and the classical ordered setting. For instance, an idempotent quasi-overlap function with the neutral element $1$ forces the underlying bounded trellis to be a bounded lattice, while finite non-trivial bounded psosets do not admit quasi-overlap functions satisfying the cancellation law or the limiting property (Propositions \ref{prop:idempotent-forces-lattice}, \ref{prop:finite-no-cancellation} and \ref{prop:finite-no-limiting}). In particular, the constructing method of quasi-overlap functions based on transitive-endpoints combines local quasi-overlap functions defined on suitable subtrellises into a global one, which takes into account the failure of transitivity and therefore differs essentially from the usual ordinal-sum constructions in lattice-valued settings (Theorem \ref{thm:gluing}).


\begin{thebibliography}{99}
\bibitem{Beliakov2007}
G. Beliakov, A. Pradera, T. Calvo, Aggregation Functions: A Guide for Practitioners, Springer, Berlin, 2007.

\bibitem{Grabisch2009}
M. Grabisch, J.L. Marichal, R. Mesiar, E. Pap, Aggregation Functions, Cambridge University Press, Cambridge, 2009.

\bibitem{Klement2000}
E.P. Klement, R. Mesiar, E. Pap, Triangular Norms, Kluwer Academic Publishers, Dordrecht, 2000.

\bibitem{Menger1942}
K. Menger, Statistical metrics, Proc. Natl. Acad. Sci. USA 28 (1942) 535--537.

\bibitem{Schweizer1983}
B. Schweizer, A. Sklar, Probabilistic Metric Spaces, North-Holland, New York, 1983.

\bibitem{Bustince2010}
H. Bustince, J. Fernandez, R. Mesiar, J. Montero, R. Orduna, Overlap functions, Nonlinear Anal. 72 (2010) 1488--1499.

\bibitem{Bustince2012}
H. Bustince, M. Pagola, R. Mesiar, E. Hüllermeier, F. Herrera, Grouping, overlaps, and generalized bientropic functions for fuzzy modeling of pairwise comparisons, IEEE Trans. Fuzzy Syst. 20 (2012) 405--415.

\bibitem{Bedregal2013}
B. Bedregal, G.P. Dimuro, H. Bustince, E. Barrenechea, New results on overlap and grouping functions, Inf. Sci. 249 (2013) 148--170.

\bibitem{Paiva2021}
R. Paiva, R. Santiago, B. Bedregal, E. Palmeira, Lattice-valued overlap and quasi-overlap functions, Inf. Sci. 562 (2021) 180--199.

\bibitem{Qiao2021}
J. Qiao, Overlap and grouping functions on complete lattices, Inf. Sci. 542 (2021) 406--424.

\bibitem{Qiao2022}
J. Qiao, Constructions of quasi-overlap functions and their generalized forms on bounded partially ordered sets, Fuzzy Sets Syst. 446 (2022) 68--92.
\bibitem{Ouyang2021}
Y. Ouyang, H.-P. Zhang, B. De Baets, Ordinal sums of triangular norms on a bounded lattice, Fuzzy Sets Syst. 408 (2021) 1--12.

\bibitem{Ertugrul2015}
Ü. Ertuğrul, F. Karaçal, R. Mesiar, Modified ordinal sums of triangular norms and triangular conorms on bounded lattices, Int. J. Intell. Syst. 30 (2015) 807--817.
\bibitem{QiaoZhao2022}
J. Qiao, B. Zhao, $\mathcal{I}_{G,N}$-implications induced from quasi-grouping functions and negations on bounded lattices, Int. J. Uncertain. Fuzziness Knowl.-Based Syst. 30 (2022) 925--949.

\bibitem{SunPangZhang2022}
Y. Sun, B. Pang, S.-Y. Zhang, Binary relations induced from quasi-overlap functions and quasi-grouping functions on a bounded lattice, Comput. Appl. Math. 41 (2022) 340.

\bibitem{Monteiro2024}
A.S. Monteiro, R. Santiago, B. Bedregal, E. Palmeira, J. Ara\'ujo, On retractions and extension of quasi-overlap and quasi-grouping functions defined on bounded lattices, J. Intell. Fuzzy Syst. 46 (2024) 863--877.

\bibitem{Skala1971}
H.L. Skala, Trellis theory, Algebra Univers. 1 (1971) 218--233.

\bibitem{Skala1972}
H.L. Skala, Trellis Theory, Mem. Amer. Math. Soc. 121, American Mathematical Society, Providence, RI, 1972.

\bibitem{Fishburn1970}
P.C. Fishburn, Intransitive indifference in preference theory: a survey, Oper. Res. 18 (1970) 207--228.

\bibitem{Kerr2002}
B. Kerr, M.A. Riley, M.W. Feldman, B.J.M. Bohannan, Local dispersal promotes biodiversity in a real-life game of rock-paper-scissors, Nature 418 (2002) 171--174.

\bibitem{Zedam2023}
L. Zedam, B. De Baets, Triangular norms on bounded trellises, Fuzzy Sets Syst. 462 (2023) 108468.

\bibitem{Geng2026}
J. Geng, R. Wang, Z. Chen, Triangular norms on bounded psosets via Adjunctions, Fuzzy Sets Syst. 532 (2026) 109814.

\bibitem{Kong2024}
Y. Kong, B. Zhao, Uninorms on bounded trellises, Fuzzy Sets Syst. 481 (2024) 108898.

\bibitem{Xiu2025}
Z. Xiu, X. Zheng, Nullnorms on bounded trellises, Fuzzy Sets Syst. 521 (2025) 109599.

\bibitem{JiangWangLiu2025}
D.-X. Jiang, Y.-M. Wang, H.-W. Liu, New constructions of nullnorms on bounded trellises, Kybernetika 61 (2025) 872--894.

\bibitem{KongLiu2026}
Y. Kong, H.-W. Liu, Several classes of uninorms on bounded pseudo-ordered sets, Fuzzy Sets Syst. 528 (2026) 109725.

\bibitem{Birkhoff73} G. Birkhoff, Lattice Theory, vol.XXV, 3rd ed., American Mathematical Society Colloquium Publications, Providence, RI, 1973.

\bibitem{Palmeira2015}
E.S. Palmeira, B. Bedregal, H. Bustince,
Applying two different methods to extend restricted dissimilarity functions,
Proc. IFSA-EUSFLAT 2015 (2015) 1126--1133.

\bibitem{WangHu2022}
Y. Wang, B.Q. Hu,
On ordinal sums of overlap and grouping functions on complete lattices,
Fuzzy Sets Syst. 439 (2022) 1--28.

\bibitem{GengZhangQiao2025}
J. Geng, Y. Zhang, J. Qiao,
Ordinal sums of quasi-overlap functions on bounded lattices,
Fuzzy Sets Syst. 521 (2025) 109581.
\end{thebibliography}
\end{document}